\documentclass[11pt]{amsart}

\usepackage[colorlinks = true,
            linkcolor = red,
            urlcolor  = magenta,
            citecolor = black,
            anchorcolor = blue]{hyperref}
\usepackage{color}
\usepackage[shortalphabetic,initials,nobysame]{amsrefs}
\usepackage{upgreek}
\usepackage{tikz}
\usepackage[applemac]{inputenc} 
\usetikzlibrary{arrows}
\usepackage{xcolor}
\usepackage[all]{xy}
\usepackage{graphicx}
\usepackage{wrapfig}
\usepackage{yfonts}
\usepackage{graphicx}
\usepackage{amssymb}
\usepackage{epstopdf}
\usepackage{mathrsfs}
\usepackage[vcentermath]{youngtab}
\usepackage[mathcal]{eucal}
\usepackage{bm}

\renewcommand{\eprint}[1]{\href{https://arxiv.org/abs/#1}{#1}}

\BibSpec{article}{%
    +{}  {\PrintAuthors}                {author}
    +{,} { \textit}                     {title}
    +{.} { }                            {part}
    +{:} { \textit}                     {subtitle}
    +{,} { \PrintContributions}         {contribution}
    +{.} { \PrintPartials}              {partial}
    +{,} { }                            {journal}
    +{}  { \textbf}                     {volume}
    +{}  { \PrintDatePV}                {date}
    +{,} { \issuetext}                  {number}
    +{,} { \eprintpages}                {pages}
    +{,} { }                            {status}
    +{,} { \DOI}                   {doi}
    +{,} { \eprint}        {eprint}
    +{}  { \parenthesize}               {language}
    +{}  { \PrintTranslation}           {translation}
    +{;} { \PrintReprint}               {reprint}
    +{.} { }                            {note}
    +{.} {}                             {transition}
    +{}  {\SentenceSpace \PrintReviews} {review}
}

\makeatletter \@addtoreset{equation}{section} \makeatother

\newcommand{\ard}{{\mathbf{d}}}

\newcommand{\Complex}{\mathbb{C}}

\makeatletter
\newcommand{\ellSN}{\mathop{\operator@font sn}\nolimits}
\newcommand{\ellCN}{\mathop{\operator@font cn}\nolimits}
\newcommand{\ellDN}{\mathop{\operator@font dn}\nolimits}
\newcommand{\ellAM}{\mathop{\operator@font am}\nolimits}
\newcommand{\ellK}{\mathop{\smash{\operator@font K}\vphantom{a}}\nolimits}
\newcommand{\ellE}{\mathop{\smash{\operator@font E}\vphantom{a}}\nolimits}
\makeatother

\ifx\genfrac\sdflkaj

\else

\fi

\newcommand{\beq}{\begin{equation}}
\newcommand{\eeq}{\end{equation}}

\makeatletter
\def\mr@ignsp#1 {\ifx\:#1\@empty\else #1\expandafter\mr@ignsp\fi}%
\newcommand{\multiref}[1]{\begingroup
\xdef\mr@no@sparg{\expandafter\mr@ignsp#1 \: }%
\def\mr@comma{}%
\@for\mr@refs:=\mr@no@sparg\do{\mr@comma\def\mr@comma{,}\ref{\mr@refs}}%
\endgroup}
\makeatother

\newcommand{\hypref}[2]{\ifx\href\asklfhas #2\else\href{#1}{#2}\fi}
\newcommand{\Secref}[1]{Section~\multiref{#1}}
\newcommand{\secref}[1]{Sec.~\multiref{#1}}
\newcommand{\Appref}[1]{Appendix~\multiref{#1}}

\newcommand{\tabref}[1]{Tab.~\multiref{#1}}

\newcommand{\figref}[1]{Fig.~\multiref{#1}}
\renewcommand{\eqref}[1]{(\multiref{#1})}

\def\[{\begin{equation}}
\def\]{\end{equation}}
\def\<{\begin{eqnarray}}
\def\>{\end{eqnarray}}

\newtheorem{theorem}{Theorem}[section]
\newtheorem{lemma}[theorem]{Lemma}
\newtheorem{proposition}[theorem]{Proposition}
\newtheorem{corollary}[theorem]{Corollary}

\newtheorem{conjecture}[theorem]{Conjecture}
\newtheorem{definition}[theorem]{Definition}

\ifx\href\asklfhas\newcommand{\href}[2]{#2}\fi

\title{A-type Quiver Varieties and ADHM Moduli Spaces}

\author{Peter Koroteev} 
\address{Department of Mathematics\\ University of California\\ Davis CA 95616\\ United States of America}
\address{Department of Mathematics\\ University of California\\ Berkeley CA 94720\\ United States of America}
\address{Kavli Institute for Theoretical Physics\\ University of California\\ Santa Barbara CA 93106\\ United States of America}
\email{pkoroteev@math.ucdavis.edu}

\begin{document}
\maketitle

\begin{abstract}
We study quantum geometry of Nakajima quiver varieties of two different types -- framed A-type quivers and ADHM quivers. While these spaces look completely different we find a surprising connection between equivariant K-theories thereof with a nontrivial match between their equivariant parameters. In particular, we demonstrate that quantum equivariant K-theory of $A_n$ quiver varieties in a certain $n\to\infty$ limit reproduces equivariant K-theory of the Hilbert scheme of points on $\mathbb{C}^2$. We analyze the correspondence from the point of view of enumerative geometry, representation theory and integrable systems. We also propose a conjecture which relates spectra of quantum multiplication operators in K-theory of the ADHM moduli spaces with the solution of the elliptic Ruijsenaars-Schneider model.
\end{abstract}

\setcounter{tocdepth}{1}
\tableofcontents

\section{Introduction}
Recent study of geometry of symplectic resolutions \cite{Braverman:2010ei,2012arXiv1211.1287M,Okounkov:2015aa,Okounkov:2018,Okounkov:2017} and, in particular, Nakajima quiver varieties \cite{Nakajima:2001qg}, has uncovered a deep relationship between enumerative geometry, representation theory and integrable systems. Using the theory of quasimaps developed in \cite{Ciocan-Fontanine:2011tg} an explicit relation between quantum K-theory of quiver varieties and quantum difference equations of Knizhnik-Zamolodchikov-type (qKZ) \cite{Frenkel:92qkz} was found \cite{Okounkov:2015aa,Okounkov:2016sya}. 

Within this framework it was possible to understand the ring structure of K-theory of Nakajima varieties as Bethe algebras and explicit formulae for quantum multiplication \cite{Pushkar:2016} as well as an explicit geometric interpretation by Aganagic and Okounkov of qKZ equations \cite{Aganagic:2017be} and their connection to integrable systems of Calogero-Ruijsenaars-Schneider type \cite{Koroteev:2018}. 
It was also demonstrated \cite{Koroteev:2017, Koroteev:2018} that the foundational results by Givental and Lee \cite{2001math8105G} in quantum K-theory of compact K\"ahler spaces, such as flag varieties, can be reproduced in this framework in the limit when the equivariant volume of cotangent fibers in the corresponding quiver variety vanishes. This limit also transforms the Ruijsenaars-Schneider model into the q-Toda chain.

In physics literature there was an independent breakthrough in understanding of the connection between integrable systems and quantum geometry. In seminal papers by Nekrasov and Shatashvili \cite{Nekrasov:2009ui,Nekrasov:2009uh} an equivalence between the spaces of solutions of Bethe Ansatz equations for XXX (XXZ) spin chains and quantum cohomology (K-theory) of A-type quiver varieties was conjectured (later it was proven in \cite{Pushkar:2016,Koroteev:2017}). In \cite{Gaiotto:2013bwa, Bullimore:2015fr} the so-called quantum/classical duality between XXZ spin chains and integrals of motion of the trigonometric Ruijsenaars-Schneider (tRS) model was formulated. This lead us to a clear understanding of quantum K-theory of the quiver varieties in question in terms of the tRS system \cite{Koroteev:2017}. 

\subsection{Goals of the Paper}
A collection of vertices and oriented edges connecting them forms a quiver. A Nakajima quiver variety can be thought of as a cotangent bundle to spaces of representations of such quivers modulo the automorphisms of vertices. Currently in the literature two types of Nakajima quiver varieties attract a significant amount of interest -- ADE-type quivers and their affine cousins, which are ubiquitous to many branches of mathematics (we shall focus solely on the A-type quivers), and Atiyah-Drinfeld-Hitchin-Manin (ADHM) quivers which arise in the study of moduli spaces of sheaves on surfaces \figref{fig:QuiverVarieryNak}. These two types of quiver varieties are typically discussed independently and the goal of the current paper is to build a bridge between geometric properties of these spaces. 
\begin{figure}[!h]
\begin{tikzpicture}[scale =1.2]
\draw [ultra thick] (0,0) -- (3,0);
\draw [ultra thick] (3,1) -- (3,0);
\draw [fill] (0,0) circle [radius=0.1];
\draw [fill] (1,0) circle [radius=0.1];
\draw [fill] (2,0) circle [radius=0.1];
\draw [fill] (3,0) circle [radius=0.1];
\node (1) at (0.1,-0.3) {$\mathbf{v}_1$};
\node (2) at (1.1,-0.3) {$\mathbf{v}_2$};
\node (3) at (2.1,-0.3) {$\ldots$};
\node (4) at (3.1,-0.3) {$\mathbf{v}_{n-1}$};
\fill [ultra thick] (3-0.1,1) rectangle (3.1,1.2);
\node (5) at (3.1,1.45) {$\mathbf{w}_{n-1}$};
\end{tikzpicture}
\qquad \qquad \qquad
\begin{tikzpicture}[xscale=1.5, yscale=1.5]
\fill [ultra thick] (1-0.1,1) rectangle (1.1,1.2);
\draw [fill] (1,0) circle [radius=0.1];
\draw [-, ultra thick] (1,1) -- (1,0.1);
\draw[-, ultra thick]
(1,0) arc [start angle=90,end angle=450,radius=.3];
\node (1) at (1.4,1.15) {$\mathscr{W}$};
\node (2) at (1.4,0.2) {$\mathscr{V}$};
\end{tikzpicture}
\caption{Left: $A_{n-1}$ quiver variety with framing on the last node (cotangent bundle to a flag variety). Right: The ADHM quiver.}
\label{fig:QuiverVarieryNak}
\end{figure}
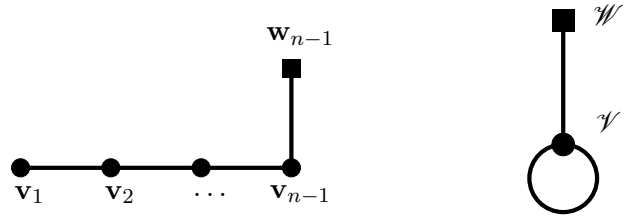 

Indeed, at first, these two quiver varieties look completely different -- the former is $A_n$-type quiver with $n$ vertices, while the latter has only one framed vertex and a loop. Yet we claim that there is a nontrivial correspondence between equivariant K-theories of these varieties. In particular, we shall demonstrate (see Theorems \ref{Th:EmbeddingTh} and \ref{Th:EmbeddingSheavesN}) the equivalence between the following spaces
\begin{itemize}
\item $\mathsf{T}$-equivariant K-theory of the moduli space of genus zero quasimaps to $A_n$-type quiver $X_n$ shown on the left in \figref{fig:QuiverVarieryNak}. 
\begin{equation}
K_{\mathsf{T}}(\textbf{QM}(\mathbb{P}^1,X_n)). \notag
\end{equation}
We shall work with $\textbf{w}_{n-1}=n$ and $\textbf{v}_{i}=i$ for $i=1,\dots,n-1$; then $X_n$ is called the cotangent bundle to \textit{complete} flag variety in $\mathbb{C}^n$.
In the above $\mathsf{T}$ is the maximal torus of $GL(n,\mathbb{C})\times \mathbb{C}_q^\times\times \mathbb{C}_\hbar^\times$, where $GL(\textbf{w}_i,\mathbb{C})\times \mathbb{C}^{\times}_\hbar$
acts as automorphisms of $X_n$ and $\mathbb{C}^\times_{\hbar}$ scales the cotangent directions with character $\hbar$\footnote{To be more precise, $\hbar$ is the class in the representation ring of the weight one representation.}, while $\mathbb{C}^\times_{q}$ acts multiplicatively on the base curve.

\item $\mathbb{C}_q^\times\times \mathbb{C}_\hbar^\times$-equivariant K-theory of the ADHM moduli space 
\begin{equation}
\bigoplus\limits_{l=0}^k K_{q,\hbar}(\mathcal{M}_{\text{ADHM}})\,. \notag
\end{equation}
Geometrically $\mathcal{M}_{\text{ADHM}}$ shown on the right of \figref{fig:QuiverVarieryNak} describes the moduli space of torsion-free sheaves of rank $N$ on $\mathbb{C}^2$, where $\mathscr{V}$ is a tautological bundle of rank $l$ and $\mathscr{W}=\mathbb{C}^N$ is trivial bundle of rank $N$. We assume that $l<n$ for all $l$. For $N=1$ this moduli space coincides with the Hilbert scheme of $l$ points on $\mathbb{C}^2$. The maximal torus $\mathbb{C}_q^\times\times \mathbb{C}_\hbar^\times$ acts naturally on $\mathbb{C}^2$ by dilations.
\end{itemize}

Then we study the limit $n\to\infty$ which will allow us to construct an embedding of the direct sum of Hilbert schemes of all points into $X_\infty$. 

Under this duality we uncover a highly nontrivial identification of the equivariant parameters: in K-theory of $A_n$ quiver $q$ was acting on the base curve, while on the ADHM side it is one of the equivariant parameters of $\mathbb{C}^2$, which is on equal footing with $\hbar$.

\subsection{Motivation}
Our construction is partly motivated by topological phase transitions and geometric engineering \cite{Gopakumar:aa} in string theory, as well as recent progress in understanding of quantum one-dimensional hydrodynamical models \cite{Koroteev:2016,Koroteev:2018a} in the context of supersymmetric gauge theories and quantum geometry. K. Costello performed somewhat similar analysis in the context of holography for M2 branes \cite{Costello:2017aa}. 

The correspondence which we are studying in this work is based on the conifold transition which connects two different desingularizations of the conifold geometry:
\begin{itemize}
\item M-theory on \textit{deformed} conifold $T^*S^3$ with $n$ M5 branes (six-dimensional objects) wrapping zero section $S^3$ and a three-manifold $M_3$.
\item M-theory on \textit{resolved} conifold $\mathscr{O}(-1)_{\mathbb{P}^1}^{\oplus 2}$. 
\end{itemize}
The former construction, in the presence of certain defects as well as flux $\hbar$ through one of the complimentary complex planes and after dimensional reduction, leads to gauge theory on $S^1\times\mathbb{C}_q$ whose quiver description is given on the left in figure \figref{fig:QuiverVarieryNak}. The latter M-theory construction sources a five-dimensional gauge theory on $S^1\times \mathbb{C}_q\times\mathbb{C}_\hbar$.  The moduli space of instantons in this 5d theory is given by the ADHM quiver on the right of \figref{fig:QuiverVarieryNak} (see \cite{Nakajima:2014aa} for details on instanton counting). The Gopakumar-Vafa topological transition occurs in the $n\to\infty$ limit when the deformed conifold is replaced by the resolved conifold. In Section 5 of \cite{Koroteev:2016} the geometric transition which is relevant for this paper is illustrated using brane constructions of Hanany-Witten type. 

In the present paper we shall examine this interesting duality from the point of view enumerative geometry, representation theory and integrable systems.

\subsection{Structure of the Paper}
In \Secref{Sec:qKZtRS} we review the construction of quantum K-theory of Nakajima quiver varieties with focus on K-theory of $X_n$. We compute equivariant K-theoretic vertex functions for those spaces and prove in Theorem \ref{Th:tRSeqproof} that they serve as the eigenfunctions of tRS Hamiltonians which act by shifts on quantum parameters. Their eigenvalues are symmetric functions of equivariant parameters $a_1,\dots, a_n$ of the maximal torus of the framing $GL(\textbf{w}_{n-1})$.

Then we make a key observation that the vertex functions for the $A_n$ quiver, once properly normalized, at specific values of the equivariant parameters \eqref{eq:aequivspec} truncate to symmetric Macdonald polynomials of quantum parameters of $X_n$. We establish this fact using difference equations which both vertex functions and Macdonald polynomials satisfy.
Therefore we can make a one-to-one correspondence between K-theory vertex functions and Young tableaux which label the Macdonald polynomials. These polynomials also describe fixed points of the maximal torus of the equivariant K-theory of Hilbert schemes of points, which is reviewed in \Secref{Sec:HIlb}. 

The symmetric Macdonald polynomials appear in modules of spherical double affine Hecke algebra (DAHA) for $\mathfrak{gl}_n$. We prove that projective $n\to\infty$ limit of these modules give rise to Fock modules of the elliptic Hall algebra, whose evaluation parameters are related to equivariant parameters of $X_n$. Further in \Secref{Sec:HIlb} we prove main Theorem \ref{Th:EmbeddingTh} about embedding of the K-theory of Hilbert scheme of points into the K-theory of the moduli space of genus zero quasimaps to $X_n$. We summarize the correspondence between K-theories of both spaces in Table \ref{Tab:Corresp}. In order to accommodate Hilbert schemes for an arbitrary large number of points we need to send the rank of the A-type quiver $n$ to infinity. 

\Secref{Sec:RankNSheaves} is devoted to a generalization of the construction which we developed in \Secref{Sec:HIlb} to the moduli spaces of torsion-free sheaves of rank $N$. Theorem \ref{Th:EmbeddingSheavesN} directly generalizes  Theorem \ref{Th:EmbeddingTh}. We demonstrate how the tensor product of Fock modules reduces to a MacMahon module once its evaluation parameters are specified to the resonance locus \eqref{eq:LocusNMacMaheval} and how to obtain these modules from K-theory of $X_\infty$. 

Finally, in \Secref{Sec:QuantumMult} we put forward a conjecture which relates operators of quantum multiplication in K-theory of the moduli spaces of framed parabolic sheaves on $\mathbb{C}^2$ with the spectrum of the elliptic Ruijsenaars-Schneider (eRS) model. Under this correspondence, which generalizes an analogous statement in quantum cohomology \cite{Negut:2011aa}, the elliptic deformation parameter coincides with the quantum parameter in the equivariant K-theory of the moduli space. We use the corresponding ADHM quiver to describe  the spectrum of the multiplication operator (see Conjecture \ref{eq:QuantMultConj}).

\subsection*{Acknowledgements}
I would like to thank I. Cherednik, E. Gorsky, S. Gukov, S. Katz, A. Negut, A. Oblomkov, A. Okounkov, P. Pushkar, A. Smirnov, Y. Soibelman, A. Zeitlin for valuable discussions and suggestions. I acknowledge support of IH\'ES and funding from the European Research Council (ERC) under the European Union's Horizon 2020 research and innovation program (QUASIFT grant agreement 677368). I would also like to thank Simons Center for Geometry and Physics in Stony Brook, Mathematical Sciences Research Institute in Berkeley, Yau Mathematical Center at Tsinghua Univeristy, Beijing \cite{talkSanya}, where part of this work was done. This research was supported in part by the National Science Foundation under Grant No. NSF PHY-1748958 and in part by AMS-Simons grant.

\section{$A$-type Quiver Varieties and Macdonald polynomials}\label{Sec:qKZtRS}
Let us first review the construction of quantum equivariant K-theory of Nakajima quiver varieties. Details can be found in \cite{Nakajima:2001qg,Negut:thesis}. Then we shall discuss in detail the vertex functions specified for quivers of type A. 

\subsection{Quiver Varieties}
Let us start with a framed quiver with set of vertices denoted by $I$. A representation of such quiver is given by a set of vector spaces $V_i$ corresponding to internal vertices and $W_i$ for the framing vertices, together with a set of morphisms between these vertices. Let $\text{Rep}(\mathbf{v},\mathbf{w})$ be the linear space of quiver representation with dimension vectors $\mathbf{v}$ and $\mathbf{w}$, where $\mathbf{v}_i=\text{dim}\ V_i$, $\mathbf{w}_i=\text{dim}\ W_i$. Let $\mu :T^*\text{Rep}(\mathbf{v},\mathbf{w})\to \text{Lie}{(G)}^{*}$, $G=\prod_i GL(V_i)$, be the moment map and let $L(\mathbf{v},\mathbf{w})=\mu^{-1}(0)$ be the zero locus of the moment map. \figref{fig:QuiverVarieryNak} provides an example of two framed quiver varieties.

The Nakajima variety $X$ corresponding to a quiver is an algebraic symplectic reduction
\begin{equation}
X=L(\mathbf{v},\mathbf{w})/\!\! /_{\theta}G=L(\mathbf{v},\mathbf{w})_{ss}/G,
\end{equation}
depending on a certain choice of stability parameter $\theta\in {\mathbb{Z}}^I$.
\cite{Ginzburg:}. Group
\begin{equation}
G=\prod GL(Q_{ij})\times\prod GL({W}_i)\times \mathbb{C}^{\times}_\hbar
\end{equation}
acts as automorphisms of $X$, coming form its action on the space of representations $\text{Rep}(\mathbf{v},\mathbf{w})$ of the quiver. Here $Q_{ij}$ is the incidence matrix of the quiver, $\mathbb{C}_{\hbar}$ scales cotangent directions with character $\hbar$ and their symplectic form with $\hbar^{-1}$. We denote by $T=\mathbb{T}(G)$ the maximal torus of $G$.

For given $X$ one can define a set of tautological bundles on it $V_i, W_i, i\in I$ as bundles constructed by assigning to each point the corresponding vector spaces $V$ and $W$. Bundles $W_i$ are topologically trivial. Tensorial polynomials of these bundles and their duals generate K-theory ring $K_{\mathsf{T}}(X)$ according to Kirwan surjectivity conjecture, which is recently shown to be true on the level of cohomology \cite{McGerty:2016kir}.

\subsection{Quasimaps to Nakajima Quiver Varieties}
First let us recall the basics of the quasimap theory to quiver varieties (see \cite{Pushkar:2016}, Sec. 2.2 or Definition 7.2.1 of \cite{Ciocan-Fontanine:2011tg}). Let us fix a rational curve $\mathcal{D}\simeq\mathbb{P}^1$ and a set of distinct points $p_1,\dots, p_m\in\mathcal{D}$.

\begin{definition}
A  stable genus zero quasimap from curve $\mathcal{C}$ to $X$ relative to points $p_1,\dots, p_m$ is given by the following data
$$
(f,\mathcal{C},P,\pi,p'_1,\dots, p'_m)
$$
where
\begin{itemize}
\item $\mathcal{C}$ is a genus zero connected curve with at worst nodal singularities and $p'_1,\dots, p'_m$ are nonsingular points of $\mathcal{C}$.
\item P is a principal $G$-bundle over $\mathcal{C}$
\item $f$ is a section of 
$
P\times_G T^* \text{Rep}(\boldsymbol{v},\boldsymbol{w})
$
satisfying $\mu =0$, where
\item $\pi: \mathcal{C}\to\mathcal{D}$ is a regular map satisfying the following conditions:
\end{itemize}
\begin{enumerate}
\item There is a distinguished component $\mathcal{C}_0$ of $\mathcal{C}$ so that $\pi$ restricts to an isomorphism: $\pi: \mathcal{C}_0\simeq \mathcal{D}$ and $\pi(\mathcal{C}\backslash \mathcal{C}_0)$ is zero-dimensional
\item $\pi(p'_i)=p_i$
\item $f(p)$ is stable for all but a finite set of points disjoint from $p'_1,\dots, p'_m$ and the nodes of $\mathcal{C}$.
\item The line bundle $\omega_{\widetilde{\mathcal{C}}}(\sum_i p_i+\sum_k q_k)\otimes\mathcal{L}^\epsilon_\theta$ is ample for every rational positive $\epsilon$ where $\mathcal{L}^\epsilon_\theta=P\times_G \mathbb{C}_\theta$, $\widetilde{\mathcal{C}}$ is the closure of $\mathcal{C}\backslash \mathcal{C}_0$, $q_k$ are the nodal points of $\widetilde{\mathcal{C}}$ and $\mathbb{C}_\theta$ is the one dimensional $G$-module defined by the stability condition.
\end{enumerate}
\end{definition}

Conditions (1)-(3) above imply that curve $\mathcal{C}$ is a union of the distinguished component $\mathcal{C}_0$ and chains of rational curves attached to it at pints $p_1,\dots,p_m$ via a nodal singularity. Projection $\pi$ collapses the chain of $\mathbb{P}^1$s attached to each point $p_i$ to that point itself. Away from those points $\pi$ is an isomorphism. Condition (4) above states that the total number of special points on each component of $\mathcal{C}$ including the nodes is at least two. This means that point $p'i$ is located on the last component of its chain.

\begin{definition}
A relative quasimap $
(f,\mathcal{C},P,\pi,p'_1,\dots, p'_m)
$ is nonsingular at $p\in\mathcal{C}$ if $f(p)$ is stable. In this case $f(p)$ is a point in the quiver variety.
\end{definition}

\begin{definition}
The degree of a quasimap $
(f,\mathcal{C},P,\pi,p'_1,\dots, p'_m)
$ is $\boldsymbol{d}=(d_i)_{i\in\mathbb{Z}}$ where $d_i$ are degrees of rank $\boldsymbol{v}_i$ vector bundles $P\times_G \mathscr{V}_i\to\mathcal{C}$.
\end{definition}

\begin{theorem}[\cite{Ciocan-Fontanine:2011tg}]
The stack $\textsf{QM}_{{\rm relative}\, p_1,\dots,p_m}^{\ard}$ parameterizing the data of stable genus zero quasimaps to $X$ is a Deligne-Mumford stack of finite type with a perfect obstruction theory.
\end{theorem}

\begin{definition}
Let $\textsf{QM}_{{\rm nonsing}\, p_1,\dots,p_m}^{\ard}$ be the stack parameterizing the data of degree $\boldsymbol{d}$ quasimap to $X$ relative to $p_1,\dots, p_m$ such that $\mathcal{C}\simeq\mathcal{D}$. 
\end{definition}

Restricting the obstruction theory of $\textsf{QM}_{{\rm relative}\, p_1,\dots,p_m}^{\ard}$ gives a perfect obstruction theory on $\textsf{QM}_{{\rm nonsing}\, p_1,\dots,p_m}^{\ard}$.

We can define an evaluation map to the quotient stack
$$
{\rm ev}_p (f,\mathcal{C},P,\pi,p'_1,\dots, p'_m) = f(p)\in[\mu^{-1}(0)/G]\,.
$$
Given a Shur functor $\tau$ in tautological bundles on $X$ we can consider the associated K-theory class on $[\mu^{-1}(0)/G]$. This allows us to define an induced K-theory class $\tau_{{\rm stack}}$ on $\textsf{QM}_{{\rm relative}\, p_1,\dots,p_m}^{\ard}$ as
$
\tau|_p = {\rm ev}_p ^*(\tau_{{\rm stack}})\,.
$

\subsection{The Quantum K-theory}
The quasimap moduli spaces have a natural action of maximal torus, lifting its action from $X$. When there are at most two special points and the base curve $\mathcal{C}$ is ${\mathbb{P}}^1$ we extend $T$ by additional torus $\mathbb{C}^{\times}_q$, which scales ${\mathbb{P}}^1$ such that the tangent space $T_{0} {\mathbb{P}}^1$ has character denoted by $q$. We shall include this action in the full torus by $\mathsf{T}=T\times \mathbb{C}^{\times}_q$. We assume that the two fixed points of $\mathbb{C}^{\times}_q$ are $p_1=0$ and $p_2=\infty$.
As explained in \cite{Okounkov:2015aa} one can construct various enumerative invariants of $X$ using virtual structure sheaves $\mathscr{O}_{\text{vir}}$ for $\textsf{QM}^{\ard}$. Using the above two marked points one can define a vertex function.
\begin{definition}
The element
\begin{equation}
V^{(\tau)}(z)=\sum\limits_{\ard=\vec{0}}^{\infty} z^{\ard} {{\rm ev}}_{p_2, *}\Big(\textsf{QM}^{\ard}_{{{\rm nonsing}} \, p_2},\widehat{{\mathscr{O}}}_{{{\rm vir}}} \tau (\left.\mathscr{V}_i\right|_{p_1}) \Big) \in  K_{\mathsf{T}}(X)_{loc}[[z]]
\label{eq:vertexQKgen}
\end{equation}
is called bare vertex with descendent $\tau\in K_T(X)$.
\end{definition}
The vertex function is an element of localized quantum K-theory of $X$. For the rest of this section we shall focus on $A$-type quiver varieties. 

\subsection{Vertex Functions for $T^*\mathbb{F}l_n$}
The cotangent bundle to the complete flag variety $X_n$ is given by A-type quiver from \figref{fig:QuiverVarieryNak} with $\mathbf{v}_i=i,\, i =1,\dots, n-1$ and $\mathbf{w}_{n-1}=n$. The details of the calculation can be found in \cite{Koroteev:2017}, here we merely present the answer.

To describe the expression for the vertex one needs to take into account  the fixed points of $\textsf{QM}^{\bf d}_{{\rm nonsing} \, p_2}$. Each such point is described by the data $(\{\mathscr{V}_i\},\{\mathscr{W}_{n-1}\})$, where ${\rm deg}\mathscr{V}_i=i, {\rm deg}\mathscr{W}_{n-1}=0$. Each bundle $\mathscr{V}_i$ can be decomposed into a sum of line bundles $\mathscr{V}_i=\mathcal{O} (d_{i,1}) \oplus \ldots \oplus \mathcal{O}(d_{i,i})$ (here $i=d_{i,1}+\ldots +d_{i,i}$). For a stable quasimap with such data to exist the collection of $d_{i,j}\geq 0$ must satisfy the following condition that
for each $i=1,\ldots ,n-2$ there should exist a subset in $\{d_{i+1,1},\ldots d_{i+1,i+1}\}$ of cardinality $i$ $\{d_{i+1,j_{1}},\ldots d_{i+1,j_i}\}$, such that $d_{i,k}\geq d_{i+1, j_k}$.
In the following we will denote the chamber containing such collections $\{d_{i,j}\}$ as $\mathrm{C}$.

\begin{theorem}[\cite{Koroteev:2017}]\label{Prop:2017paper}
Let $\textbf{p}=\mathbf{V}_{1}\subset \ldots \subset \mathbf{V}_{n-1}\subset \{a_1,\cdots,a_{n}\}$ be a chain of subsets defining a torus fixed point $\textbf{p}\in X_n^{\mathsf{T}}$.
Then the coefficient of the vertex function with descendent $\tau\in K_T(X_n)$ for this point is given by:
\begin{equation}
V^{(\tau)}_{\textbf{p}}(z) = \sum\limits_{d_{i,j}\in C}\, \left(z^\sharp\right)^{\textbf{d}} \, EHG\ \ \tau(x_{i,j} q^{-d_{i,j}}),
\label{eq:Vtauz}
\end{equation}
where $\textbf{d}=(d_1,\ldots ,d_{n-1}),d_i=\sum_{j=1}^{i}d_{i,j}$,
\beq
E=\prod_{i=1}^{n-1}\prod\limits_{j,k=1}^{i}\{x_{i,j}/x_{i,k}\}^{-1}_{d_{i,j}-d_{i,k}},\quad 
G=\prod\limits_{j=1}^{n-1} \prod\limits_{k=1}^{n} \{x_{n-1,j}/a_k\}_{d_{n-1,j}},\nonumber
\eeq
$$
H=\prod_{i=1}^{n-2}\prod_{j=1}^{i}\prod_{k=1}^{i+1}\{x_{i,j}/x_{i+1,k}\}_{d_{i,j}-d_{i+1,k}}.
$$
Here
\beq
\{x\}_{d}=\dfrac{(\hbar/x,q)_{d}}{(q/x,q)_{d}} \, (-q^{1/2} \hbar^{-1/2})^d, \ \ \textrm{where}  \ \ (x,q)_{d}=\prod^{d-1}_{i=0}(1-q^ix).\nonumber
\eeq
\end{theorem}

The coefficients of the vertex functions are (generalizations of) q-hypergeometric functions. For our purposes it will be more convenient to use the following parameterization for the quantum parameters in $K_T(X_n)$\footnote{$\zeta_1$ is also redefined to absorb factor $(-q^{1/2}\hbar^{-1/2})^{d_1}$, which arises from $E$ function above.}
\begin{align}
z^\sharp_1&=\frac{\zeta_{1}}{\zeta_2}\,,\cr
z_i^\sharp&=\frac{\zeta_{i}}{\zeta_{i+1}}\,,\quad i=2,\dots, n-2\cr
z_{n-1}^\sharp&=\frac{\zeta_{n-1}}{\zeta_{n}}\,.
\end{align}

\begin{corollary}
For $X$ as above the vertex function coefficient for the identity class $\tau=1$ reads
\begin{equation}
V^{(1)}_{\textbf{p}}(z) = \sum\limits_{d_{i,j}\in C} \prod_{i=1}^{n-1} \left(t\frac{\zeta_{i}}{\zeta_{i+1}}\right)^{d_i} \prod\limits_{j,k=1}^{i}\frac{\left(q\frac{x_{i,j}}{x_{i,k}},q\right)_{d_{i,j}-d_{i,k}}}{\left(\hbar\frac{x_{i,j}}{x_{i,k}},q\right)_{d_{i,j}-d_{i,k}}}\cdot\prod_{j=1}^{i}\prod_{k=1}^{i+1}\frac{\left(\hbar\frac{x_{i+1,k}}{x_{i,j}},q\right)_{d_{i,j}-d_{i+1,k}}}{\left(q\frac{x_{i+1,k}}{x_{i,j}},q\right)_{d_{i,j}-d_{i+1,k}}}\,,
\label{eq:V1pdef}
\end{equation}
where we define $x_{n,k}=a_k$ and $t=\frac{q}{\hbar}$.
\end{corollary}

Vertex functions $V^{(\tau)}$ can be regarded as classes in equivariant K-theory of the moduli space of quasimaps which we shall denote
\begin{equation}
\mathcal{H}_n:=K_{\textsf{T}}(\textbf{QM}(\mathbb{P}^1,X_n))
\label{eq:Kthquasimaps}
\end{equation}
for extended maximal torus $\textsf{T}$.

Now we shall demonstrate that coefficient functions of K-theory vertices obey Macdonald difference equations.

\subsection{Macdonald Difference Operators}
It was also proven in \cite{Koroteev:2018} that K-theoretic vertex functions of cotangent bundles to flag varieties satisfy tRS difference equations in equivariant parameters. In this paper we shall need the `mirror' (or symplectic dual) version of that theorem which states that properly normalized vertex functions also obey tRS difference equations in quantum parameters $\zeta_1,\dots,\zeta_n$. First let us introduce the difference operators.

\begin{definition}
The difference operators of trigonometric Ruijsenaars-Schneider model are given by
\begin{equation}
T_r(\boldsymbol{\zeta})=\sum_{\substack{\mathcal{I}\subset\{1,\dots,n\} \\ |\mathcal{I}|=r}}\prod_{\substack{i\in\mathcal{I} \\ j\notin\mathcal{I}}}\frac{\hbar\,\zeta_i-\zeta_j}{\zeta_i-\zeta_j}\prod\limits_{i\in\mathcal{I}}p_k \,,
\label{eq:tRSRelationsEl}
\end{equation}
where $\boldsymbol{\zeta}=\{\zeta_1,\dots, \zeta_n\}$, the shift operator $p_k f(\zeta_k)=f(q\zeta_k)$.
\end{definition}

Now we shall prove that the K-theory vertex function after normalization is the eigenfunction of the tRS difference operators.
\begin{theorem}\label{Th:tRSeqproof}
Let $V^{(1)}_{\textbf{p}}$ be the coefficient for the vertex function for $X$ given in \eqref{eq:V1pdef}. Define 
\begin{equation}
\mathsf{V}^{(1)}_{\textbf{p}}= \prod\limits_{i=1}^n\frac{\theta(\hbar^{i-n}\zeta_i,q)}{\theta(a_i\zeta_i,q)}\cdot V^{(1)}_{\textbf{p}}\,,
\label{eq:DefVvertMac}
\end{equation}
where $\theta(x,q)=(x,q)_\infty (qx^{-1},q)_\infty$ is basic theta-function.
Then $\mathsf{V}_{\textbf{p}}$ are eigenfunctions for tRS difference operators \eqref{eq:tRSRelationsEl} for all fixed points $\textbf{p}$
\begin{equation}
T_r(\boldsymbol{\zeta}) \mathsf{V}^{(1)}_{\textbf{p}} = e_r (\mathbf{a}) \mathsf{V}^{(1)}_{\textbf{p}}\,, \qquad r=1,\dots, n\,,
\label{eq:tRSEigenz}
\end{equation}
where $e_r$ is elementary symmetric polynomial of degree $r$ of $a_1,\dots, a_n$\,.
\end{theorem}

In order to understand the proof we shall use the integral formula for the vertex function. Using Theorem 4.8 from \cite{Koroteev:2018} we can write vertex \eqref{eq:V1pdef} as follows
\begin{equation}
V^{(1)}_{\textbf{p}} = \frac{e^{\frac{\log\zeta_n\cdot\log a_1\cdots a_n}{\log q}}}{2\pi i}\int\limits_{C_{\textbf{p}}} \prod_{m=1}^{n-1}\prod_{i=1}^{m} \frac{ds_{m,i}}{s_{m,i}} E(s_{m,i})\,\, e^{-\frac{\log \zeta_{m}/\zeta_{m+1} \cdot\log s_{m,i}}{\log q}} \cdot \prod_{j=1}^{m+1}H_{m,m+1}\left(s_{m,i},s_{m+1,j}\right)\,,
\label{eq:GenerictRSSolution}
 \end{equation}
where contour $C_{\textbf{p}}$ surrounds poles corresponding to the fixed point $\textbf{p}$ of the maximal torus of $X_n$
and the functions in the integrand are given by
\begin{equation}
H_{m,m+1}(\textbf{s}_{m},\textbf{s}_{m+1})=\prod\limits_{i=1}^{\textbf{v}_m}\prod\limits_{j=1}^{\textbf{v}_{m+1}}\dfrac{\varphi\left(\frac{q}{\hbar} \frac{s_{i,k}}{ s_{i+1,j} }\right)}{\varphi\left(\frac{s_{i,k}}{ s_{i+1,j} }\right)},
\label{eq:HyperGen}
\end{equation}
corresponding to the contribution of Hom$(\mathscr{V}_m,\mathscr{V}_{m+1})$ and 
\begin{equation}
E(\textbf{s}_n)=\prod\limits_{j,k=1}^{\mathbf{v}_n} \dfrac{\varphi\Big( \frac{s_{n,j}}{ s_{n,k}}\Big)}{\varphi\Big(t \frac{s_{n,j}}{ s_{n,k}}\Big)}\,,
\end{equation}
emerging from Hom$(\mathscr{V}_m,\mathscr{V}_m)$ in the localization computation, and where 
\begin{equation}
\varphi(x)=\prod^{\infty}_{i=0}(1-q^ix)\,.
\label{eq:PhiDef}
\end{equation}

\begin{proof}[Proof of Theorem \ref{Th:tRSeqproof}]
The verification of \eqref{eq:tRSEigenz} can be performed directly with the help of certain identities which arise from shifting of the integration contour in \eqref{eq:GenerictRSSolution}. An example of this calculation is given in \Appref{Sec:tRSP1} for $T^*\mathbb{P}^1$ and the calculation can be generalized to $T^*\mathbb{F}l_n$.

\vskip.1in
Here, however, we would like to give a complementary proof which uses previous results \cite{Koroteev:2017} and some elementary calculations. 

According to \cite{Koroteev:2017} tRS momenta $p_i$ correspond to multiplication by class $\widehat{\Lambda^i \mathscr{V}_i}\otimes\widehat{\Lambda^{i+1}\mathscr{V}^\ast_{i+1}}$ in $K_T(X_n)$, where $\mathscr{V}_i$ is the $i$-th tautological bundle over $X_n$, and are given by the following ratio of products of the corresponding Chern roots
\begin{equation}
p_i = \frac{s_{i+1,1}\cdot\dots\cdot s_{i+1,i+1}}{s_{i,1}\cdot\dots\cdot s_{i,i}},\,\qquad i=1,\dots, n-1\,.
\label{eq:pimomenta}
\end{equation}
Recall that $s_{n,i}=a_i$.
Using this fact and the definition of tRS operators \eqref{eq:tRSRelationsEl} we can define new quantum classes $\mathsf{V}^{(T_r)}_{\textbf{p}}$ for $r=1,\dots,n$. We can refer to them as \textit{tRS classes} (see \cite{Dinkins:2020nzu} for further development).

Notice that by acting with the tRS operators on the vertex function in the integral form \eqref{eq:GenerictRSSolution} we get
\begin{equation}
T_r(\boldsymbol{\zeta}) \mathsf{V}^{(1)}_{\textbf{p}} = \mathsf{V}^{(T_r)}_{\textbf{p}} \,,
\label{eq:tRSclass}
\end{equation}
where on the right we have a vertex function with descendant class $T_r$ in which $p_i$ are given by \eqref{eq:pimomenta}.
In particular, consider tRS Hamiltonian $T_n$ which according to \eqref{eq:tRSRelationsEl} reads $T_n=p_1\cdots p_n$.
From \eqref{eq:GenerictRSSolution} we immediately see that the corresponding tRS class is given by
$$
\mathsf{V}^{(T_n)}_{\textbf{p}}=a_1\cdots a_n \mathsf{V}^{(1)}_{\textbf{p}} = e_n(\textbf{a})\mathsf{V}^{(1)}_{\textbf{p}}\,.
$$
Indeed, the shift by $p_n$ of the exponential in front of the integral in \eqref{eq:GenerictRSSolution} returns $a_1\cdots a_n$ while all other shifts of $\zeta_m$ variables in the integrand cancel each other as they appear only as ratios $\zeta_m/\zeta_{m+1}$.

In other words, $\mathsf{V}^{(1)}_{\textbf{p}}$ is an eigenvector of $T_n$.
Since all tRS Hamiltonians commute with each other $[T_r,T_s]=0$, they share the set of eigenvectors; therefore, $\mathsf{V}^{(1)}_{\textbf{p}}$ is an eigenvector for all $T_r(\boldsymbol{\zeta})$, $r=1,\dots,n$ as well. It remains to be shown that the eigenvalues of the Hamiltonians $T_r(\boldsymbol{\zeta})$ are precisely given by the elementary symmetric functions $e_r(\boldsymbol{a})$.

In \cite{Koroteev:2017} (Theorem 3.4) it was proven that the eigenvalues of the multiplication operator by a quantum class $\widehat{\tau}$ in quantum K-theory of $X_n$ is given by $\tau(\textbf{s})$, where Chern roots $\textbf{s}$ of the corresponding virtual bundle solve the XXZ Bethe Ansatz equations for $X_n$ with $\textbf{s}$ playing the role of Bethe roots. It was also proven in \textit{loc. cit.} (Theorem 4.5) that these Bethe equations are equivalent to the classical tRS equations $T_r(\boldsymbol{\zeta})=e_r(\textbf{a})$. In other words, the eigenvalues of the multiplication by quantum tRS classes are precisely the corresponding elementary symmetric functions $e_r(\textbf{a})$.

%
\end{proof}

\subsection{Macdonald Polynomials}
It is well known that Macdonald polynomials (see \cite{MacDonald:286472} for review) are also eigenfunctions of tRS difference operators (Macdonald operators) \eqref{eq:tRSEigenz}. 
Indeed, for a given Young tableau $\lambda$ whose columns we denote by $\lambda_i,\,i = 1,\dots,n$, we can construct  
symmetric Macdonald polynomials of $n$ variables $P_\lambda(\boldsymbol{\zeta};q,\hbar)$. We will always assume that $\lambda$ has $n$ columns, if necessary we shall complete the diagram with empty columns.

In particular, we have
\begin{equation}
T_1(\boldsymbol{\zeta}) P_\lambda(\boldsymbol{\zeta};q,\hbar) = \left(\sum\limits_{i=1}^n q^{\lambda_i}\hbar^{i-n}\right) P_\lambda(\boldsymbol{\zeta};q,\hbar)\,.
\label{eq:tRSEigenzMac}
\end{equation}

We shall demonstrate momentarily that $P_\lambda(\boldsymbol{\zeta};q,\hbar)$ can be obtained from infinite series \eqref{eq:DefVvertMac} by specifying equivariant parameters $a_1,\dots,a_n$  which take values in a discrete lattice spanned by $q$ and $\hbar$.

\begin{proposition}
\label{Prop:Truncation}
Consider coefficient functions for K-theory of $\textsf{QM}$ to $X_n$  \eqref{eq:DefVvertMac} for all fixed points of the maximal torus. Let $\lambda$ be a partition of $k$ elements of length $n$ and $\lambda_1\geq\dots\geq\lambda_n$. Let  
\begin{equation}
\frac{a_{i+1}}{a_{i}}=q^{\ell_{i}}\hbar\,,\quad \ell_i = \lambda_{i+1}-\lambda_{i}\,, \quad i=1,\dots,n-1\,.
\label{eq:truncationa}
\end{equation}
Then there exists a fixed point $\textbf{q}$ for which
\begin{equation}
\mathsf{V}_{\textbf{q}} = P_\lambda(\boldsymbol{\zeta};q,\hbar)\,.
\end{equation}
\end{proposition}

\begin{proof}
Macdonald polynomials (independent on their normalization) satisfy difference equations of the form \eqref{eq:tRSEigenz}. 
By plugging \eqref{eq:truncationa} into the formula for the coefficient of the vertex function \eqref{eq:V1pdef} we observe that the series in $\zeta_i$ truncate at orders $\lambda_{i}$ thereby turning it into a polynomial. This polynomial is unique up to normalization.
\end{proof}

We can also rewrite $\eqref{eq:truncationa}$ as 
\begin{equation}
a_i = a q^{\lambda_i}\hbar^{i-n}\,,\quad i=1,\dots,n\,, 
\label{eq:aequivspec}
\end{equation}
where $a\in\mathbb{C}^\times$ is an arbitrary parameter, which later will be interpreted as an evaluation parameter for a Fock module, and $\lambda_{i} = \sum_{j=1}^i\ell_{n-j+1}$.
Note also that our normalization of Macdonald polynomials is slightly different from \cite{MacDonald:286472}.

\vspace{1mm}
\textbf{Remark.}
Note that out $n!$ fixed points of $T$ acting on $X_n$ only for single point $\textbf{q}$ we get a polynomial out of $V_{\textbf{q}}$. Other coefficient functions $V_{\textbf{p}}$ remain infinite series which we shall further ignore. One can verify by examining \eqref{eq:aequivspec} that in the $n\to\infty$ limit these coefficient functions will be suppressed. A choice of fixed point $\textbf{q}$ following the large-$n$ limits corresponds to zooming into a certain asymptotic region in which $|a_{\mathfrak{S}(1)}|\gg|a_{\mathfrak{S}(2)}|\gg\dots\gg |a_{\mathfrak{S}(n)}|$, where $\mathfrak{S}\in S_n$ is a permutation corresponding to the choice of fixed point $\textbf{q}$.
\begin{figure}[!h]
\includegraphics[scale=0.28]{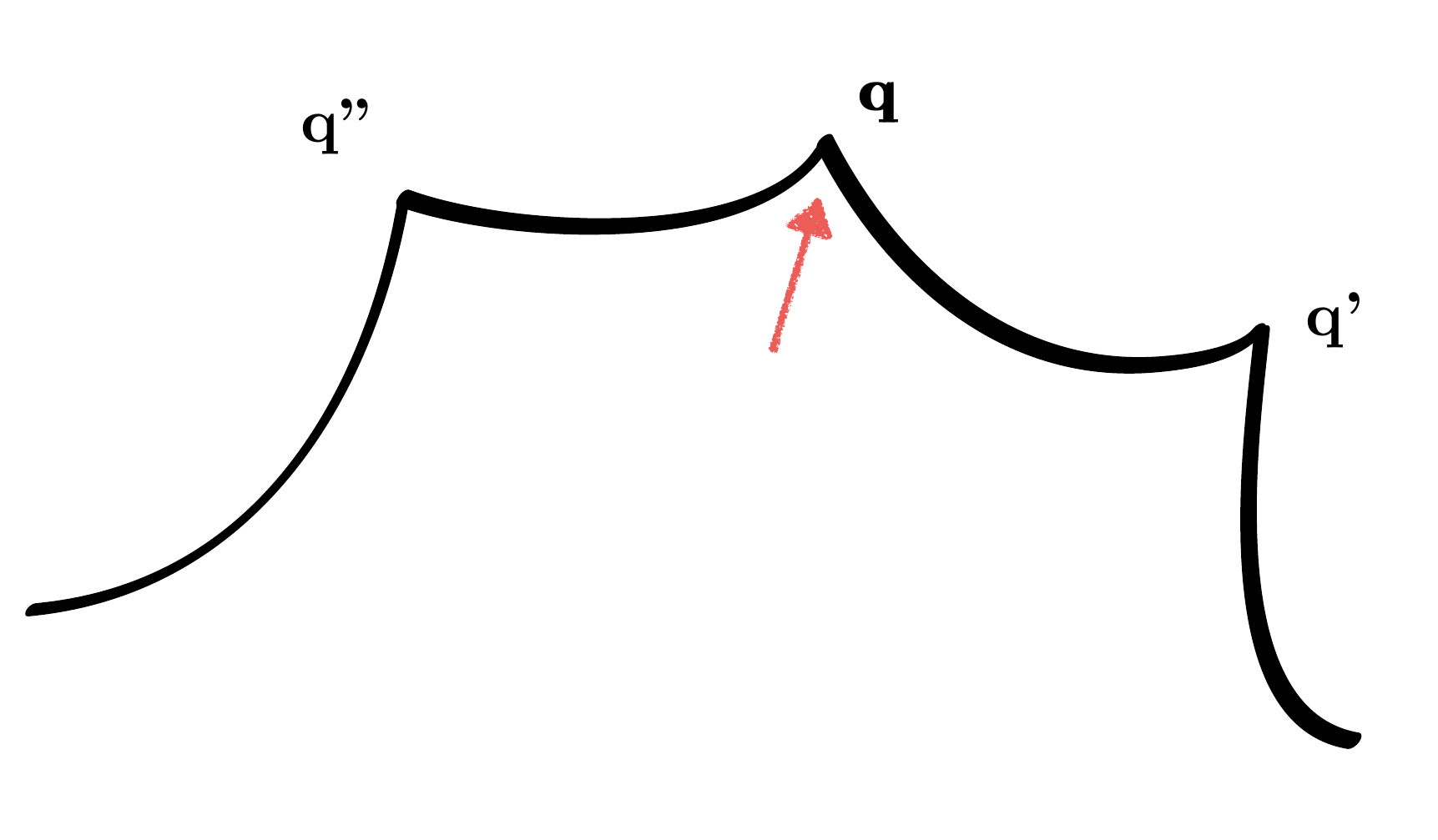}
\caption{Asymptotic regions of the space of $GL(n)$ equivariant parameters. In the $n\to\infty$ limit, provided that (\ref{eq:aequivspec}) holds, we approach the given fixed point.}
\label{fig:asympt}
\end{figure}

\vspace{1mm}
\textbf{Remark.}
The above theorem has an interesting interpretation from the point of view of enumerative geometry viewpoint and Gromov-Witten theory in particular. Formulae like \eqref{eq:V1pdef}, which are closed cousins of Givental J-functions, serve as generating functions of the equivariant gravitational descendants of $X_n$. Due to \eqref{eq:aequivspec} the q-hypergeometric-type series truncate to polynomials thereby implying that these equivariant gravitational descendants vanish identically starting from some order in $\zeta_i$'s. The  parameters of the problem get adjusted in such a way that the equivariant volume of the moduli spaces of quasimaps of degrees $(d_1,\dots,d_{n-1})$ which exceed $(\lambda_1,\dots,\lambda_{n-1})$ become zero.\footnote{I thank A. Okounkov and S. Katz for interesting discussions on these matters.}.

\vspace{1mm}
\textbf{Example.}
Let us look at the simplest example $X_2=T^*\mathbb{P}^1$. The corresponding vertex function is given by the q-hypergeometric function
\begin{equation}
V^{(1)}_{\textbf{p}}=
\sum_{d>0} \left(\frac{\zeta_1}{\zeta_2}\right)^d\,\prod_{i=1}^2
\frac{\left(\frac{q}{\hbar}\frac{a_\textbf{p}}{a_i};q\right)_d}{\left(\frac{a_\textbf{p}}{a_i};q\right)_d}= _2\!\!\phi_1\left(\hbar,\hbar\frac{a_{\textbf{p}}}{a_{\bar{\textbf{p}}}},q\frac{a_{\textbf{p}}}{a_{\bar{\textbf{p}}}};q;\frac{q}{\hbar}\frac{\zeta_1}{\zeta_2}\right)\,.
\label{eq:TP1V}
\end{equation}
Here $a_{\bar{\textbf{1}}}=a_{\textbf{2}}$ and $a_{\bar{\textbf{2}}}=a_{\textbf{1}}$. One can easily see that Weyl reflection of $\mathfrak{sl}_2$ interchanges $V^{(1)}_{\textbf{1}}$ and $V^{(1)}_{\textbf{2}}$. Then we have
\begin{equation}
\mathsf{V}_{\textbf{p}}=\frac{\theta(\zeta_1,q)}{\theta(a_1\zeta_1,q)}\frac{\theta(\hbar\zeta_2,q)}{\theta(a_2\zeta_2,q)}\cdot V^{(1)}_{\textbf{p}}\,.
\end{equation}
The condition \eqref{eq:truncationa} reads $a_1=a q^{\lambda_1}\hbar^{-1}\,,a_2= a q^{\lambda_2}$. Then we have the following Macdonald polynomials for some simple tableaux. 
\begin{align}
\mathrm{V}_{{\tiny\yng(1)}}&=\zeta _1+\zeta _2,\cr
\mathrm{V}_{{\tiny\yng(1,1)}}&=\zeta _1^2+\zeta _2^2+\frac{ (q+1) (\hbar -1)}{q \hbar -1}\zeta _1 \zeta _2,\cr
\mathrm{V}_{{\tiny\yng(1,1,1)}}&=\zeta _1^3+\zeta
   _2^3+\frac{ \left(q^2+q+1\right) (\hbar -1)}{q^2 \hbar -1}\zeta _2 \zeta _1^2+\frac{
   \left(q^2+q+1\right) (\hbar -1)}{q^2 \hbar -1}\zeta _2^2 \zeta _1\,.
\end{align}
Here the Young tableaux at the subscripts have their heights equal to $\lambda_2$. Without loss of generality we can assume $\lambda_n=0$ for all formulae. In other words, for $T^*\mathbb{P}^1$ the corresponding Macdonald polynomials depend only on one integer parameter (so-called Roger polynomials\footnote{To actually get Roger polynomials of one variable $x$ one needs to put $\zeta_1=\xi$ and $\zeta_1=\xi^{-1}$ and multiply the resulting expression by $\xi^{\lambda_2}$}). Otherwise the resulting polynomial will simply be a product of the corresponding Macdonald polynomial and a simple factor depending on $\lambda_n$. For instance
\begin{equation}
\mathrm{V}_{{\tiny\yng(1,2)}}=\zeta_2\zeta _1+\zeta^2_2=\zeta_2 \mathrm{V}_{{\tiny\yng(1)}}\,.
\end{equation}

Thus we have shown that classes \eqref{eq:DefVvertMac} of equivariant K-theory of the moduli space of quasimaps from $\mathbb{P}^1$ to $T^*\mathbb{F}l_n$ upon specification of equivariant parameters of the maximal torus of $\mathfrak{gl}_n$ according to \eqref{eq:aequivspec} reduce to symmetric Macdonald polynomials of $n$ variables. It is well known in the literature that Macdonald polynomials appear in K-theory of Hilbert scheme of points on $\mathbb{C}^2$. Our goal in the next section is to provide a precise relationship between K-theories of these two types of Nakajima quiver varieties.

\subsection{Geometric Meaning of the tRS Model}
Macdonald operators $T_1,\dots, T_n$ \eqref{eq:tRSRelationsEl} form a maximal commuting subalgebra in spherical double affine Hecke algebra (DAHA) for $\mathfrak{gl}_n$ \cite{10.2307/2118632,Cherednik:1990054} which we shall call $\mathfrak{A}_n$ here for short. It was shown by Oblomkov \cite{A.Oblomkov:, Oblomkov:cub} that this algebra arises as deformation quantization of the moduli space of $GL(n,\mathbb{C})$ flat connections on a single-punctured torus $\mathcal{M}_{\text{flat}}(GL(n,\mathbb{C}), T^2\backslash\{\text{pt}\})$\footnote{This moduli space is a hyperK\"ahler manifold and the quantization is given in complex structure $J$.} (also known as the Calogero-Moser space). Parameter $q$ is identified with the quantization parameter and $\hbar$ is the sole distinct eigenvalue of the $GL(n,\mathbb{C})$ monodromy matrix. Most of work in the literature is done on full DAHA, whereas here we are interested in its spherical part\footnote{It is spherical DAHA which often appears in physics applications, albeit with some exceptions \cite{KorGukNavSab:DAHA}}. Let us first describe the geometry behind the tRS integrable system.

The classical limit $q\to 1$ of tRS equations \eqref{eq:tRSEigenz} has a direct interpretation in terms of the moduli space $\mathcal{M}_{\text{flat}}(GL(n,\mathbb{C}), T^2\backslash\{\text{pt}\})$, which is identified with the phase space of the tRS model. This space can be described by the holonomy matrices $A$ and $B$ around the two cycles of the torus, which obey the fundamental group relation of the punctured torus 
\begin{equation}
ABA^{-1}B^{-1}=C\,,
\label{eq:ABCflat}
\end{equation}
where $C=\text{diag}(\hbar,\dots, \hbar, \hbar^{1-n})$ is the holonomy matrix around the puncture. The above flatness condition can be solved \cite{Gaiotto:2013bwa} in terms of Fenchel-Nielsen coordinates $(\zeta_i,p_i),\,i=1,\dots,n$ on the moduli space, and in the basis where matrix $A=\text{diag}(a_1,\dots,a_n)$ is diagonal, Casimir invariants of $B$ reproduce the tRS Hamiltonians \eqref{eq:tRSRelationsEl}
\begin{equation}
T_r(\boldsymbol{\zeta}) = \text{Tr}_{\Lambda^r} B\,,\qquad e_r(\vec{a}) = \text{Tr}_{\Lambda^r} A\,,
\end{equation}
while the Casimirs of $A$ are merely the symmetric functions for the equivariant parameters. 

The equality between the left and right hand sides of \eqref{eq:tRSEigenz} is the consequence of the $SL(2,\mathbb{Z})$ symmetry of $\mathcal{M}_{\text{flat}}(GL(n,\mathbb{C}), T^2\backslash\{\text{pt}\})$. Under the $S$ transformation of $SL(2,\mathbb{Z})$ holonomies  $A$ and $B$ get interchanged. One can immediately see that the defining equation \eqref{eq:ABCflat} transforms into itself if in addition we invert matrix $C$. In order to establish \eqref{eq:tRSEigenz} one needs to consider the following space $\mathcal{M}_L\times\mathcal{M}_R$ which is a Cartesian product of two copies of the moduli space of flat connections, with local coordinates $(\zeta_i,p_i,\hbar)$ and $(a_i,p^{!}_i,\hbar^{-1})$ respectively. 

\begin{proposition}
Consider the product of two tRS phase spaces $\mathcal{M}_L\times\mathcal{M}_R$ as described above which is equipped with the following symplectic form 
\begin{equation}
\Omega_{L,R} = \sum_{i=1}^n \frac{d\zeta_i}{\zeta_i}\wedge\frac{dp_i}{p_i} - \sum_{i=1}^n \frac{da_i}{a_i}\wedge\frac{dp^{!}_i}{p^{!}_i} \,.
\end{equation}
Let $\mathcal{L}_a$ be a Lagrangian submanifold in $\mathcal{M}_L\times\mathcal{M}_R$ which is given by specifying the eigenvalues of matrix $A$ in \eqref{eq:ABCflat}. Then this Lagrangian is defined by the tRS equations of motion \eqref{eq:tRSEigenz} and symplectic form $\Omega_{L,R}$ vanishes on $\mathcal{L}_a$.
\end{proposition}

We refer the reader to \cite{Bullimore:2015fr}, Sec. 2.2, for more details. There is also a symplectic dual version of this proposition for Lagrangian $\mathcal{L}_\zeta\subset \mathcal{M}_L\times\mathcal{M}_R$ which is described by specifying the eigenvalues of $B$. This leads to the dual tRS model, in which where Hamiltonians act on the equivariant parameters and the their eigenvalues are symmetric functions of quantum parameters \cite{Koroteev:2018}.

Quantization of the tRS integrable system is the achieved by imposing $\zeta_i p_j = q^{\delta_{ij}}p_j \zeta_i$.

\subsection{Spherical DAHA and Elliptic Hall Algebra}
For the purposes of this work we will not need complete description of DAHA, nevertheless let us briefly describe its structure. Full (non-spherical) $\mathfrak{gl}_n$ DAHA is a $\mathbb{C}(q,\hbar)$ associative algebra which is formed by combining two copies of affine Hecke algebras $\{t_i,X_i\}$ and $\{t_i,Y_i\}$, $i=1,\dots, n$ modulo certain relations (see \cite{10.2307/2118632} for review). Here $t_i$ are elements of $\mathfrak{gl}_n$ Hecke algebra which obey braid group relations.
Spherical subalgebra is obtained by conjugating all elements of DAHA by the total idempotent $S$. It can be shown \cite{Cherednik:1990054,Schiffmann:2008aa} that for $l>0$
\begin{equation}
Z^n_{0,l}= S\sum_i Y_i^l\, S\,,\quad Z^n_{0,-l}= q^l S\sum_i Y_i^l\, S\,,\quad Z^n_{l,0}= q^l S\sum_i X_i^l\, S\,,\quad Z^n_{-l,0}=S\sum_i X_i^l\, S\,
\label{eq:SpericalDAHAGen}
\end{equation}
generate $\mathfrak{A}_n$. In particular, the tRS operators $T_r$ are symmetric functions of $Y_i$ generators of DAHA
\begin{equation} 
T_r = e_r(Y_1,\dots, Y_n)\,.
\end{equation}
Geometrically, in terms of K-theory of $X_n$, we can define classes $\mathsf{V}^{(Y_i)}_{\textbf{p}}$ whose symmetrization yields tRS classes \eqref{eq:tRSclass}.

In the above formula $Z^n_{0,l}$ represent Macdonald operators $T_l$ \eqref{eq:tRSRelationsEl}, $Z^n_{\pm l,0}$ are  multiplication operators by a symmetric polynomial of $n$ variables of degree $l$. Elements $Z^n_{0,-l}$ can be diagonalized in the basis of Macdonald polynomials and their eigenvalues can be readily computed \cite{Cherednik:1990054}.

Let us define
\begin{equation}
\Lambda_n := \mathbb{Q}(q,\hbar)[\zeta_1,\dots,\zeta_n]^{S_n}
\end{equation}
to be the space symmetric polynomials over $\mathbb{Q}(q, \hbar)$ of n-variables. 
Algebra $\mathfrak{A}_n$ acts on the span of symmetric Macdonald polynomials $P_{\lambda}\in \Lambda_n$. Macdonald operators $Z^n_{0,l}$ act on polynomials $P_{\lambda}$ by multiplication by symmetric polynomial of equivariant parameters $s_l(a_1,\dots,a_n)$ evaluated at locus \eqref{eq:aequivspec}. 
Coefficient vertex functions \eqref{eq:V1pdef} evaluated at \eqref{eq:aequivspec} therefore can be treated as highest weight vectors for $\mathfrak{A}_n$ modules over $\Lambda_n$. Unless we further specify the values of $q$ and $\hbar$ these modules, which we shall refer to as $M_n$, are infinite-dimensional and contain Macdonald polynomials parameterized by Young tableaux $\lambda=\{\lambda_1,\dots,\lambda_n\}$ with $n$ columns. Proper combinations of generators \eqref{eq:SpericalDAHAGen} of $\mathfrak{A}_n$ will act on the states in the highest weight module by adding to or removing boxes from tableau $\lambda$ thereby producing a different Macdonald polynomial. The corresponding explicit formulae have only been worked out for $\mathfrak{A}_n$ of lower ranks \cite{Kirillov:,Suzuki:}.

In a parallel development Varagnolo and Vasserot \cite{Varagnolo:2009} studied the action of $\mathfrak{A}_n$ on K-theory of affine flag varieties. We expect to find a direct correspondence between their results and our approach.

For certain values of $q$ and $\hbar$ modules $M_n$ may become reducible and develop proper submodules. The submodule and the corresponding finite-dimensional factor module together with $M_n$ form a short exact sequence. Representation theory of $\mathfrak{A}_n$, as well as DAHA itself are rather involved and at the present moment not fully understood even for algebras of low rank. We may expect some simplifications while working with the spherical subalgebra due to $S_n$ invariance. However, already for $\mathfrak{sl}_2$ spherical DAHA, classification of finite-dimensional representations is extremely nontrivial \cite{KorGukNavSab:DAHA}. 

In the current work we will be interested in the limit $n\to\infty$ of algebra $\mathfrak{A}_n$ and its modules. Such stable limit of DAHA is known in the literature and goes by different names -- elliptic Hall algebra, Ding-Iohara-Miki algebra, Drinfeld double of shuffle algebra, quantum toroidal algebra $U_{t_1,t_2}(\hat{\hat{\mathfrak{gl}_1}})$ (see \cite{Negut:2012aa} for references therein and for a side-by-side comparison). We denote this algebra by $\mathfrak{E}$.
Thanks to the recent progress we have a fair understanding of the representation theory of $\mathfrak{E}$ including its various modules -- Fock and MacMahon modules. In \cite{Schiffmann:2009} it was shown that there is a surjective algebra homomorphism $\mathfrak{E}\to\mathfrak{A}_n$. In this paper we shall provide a map between modules of these algebras.

\subsection{Highest weight modules of $\mathfrak{A}_n$}
Consider the highest weight module 
\begin{equation}
M_n=\left\{P_{\lambda}(\boldsymbol{\zeta}|q,\hbar)\,\Big| \ell(\lambda)< n\right\}\,,
\label{eq:Mnmodule}
\end{equation}
where $\ell(\lambda)$ is the number of columns of tableau $\lambda$ in more detail. If $\lambda$ has $k$ boxes and $k<n$ then we assume the remaining columns to be empty $\{\lambda_1,\dots,\lambda_{\ell(\lambda)},0,\dots,0\}$. The module has a natural grading by the number of boxes $|\lambda|=k$ 
\begin{equation}
M_n = \bigoplus_{k} M_n^k\,.
\end{equation}
By construction when $k\leq n$ the component $M_n^k$ will contain $P_\lambda$'s for all possible young tableau of $k$ boxes, whereas for $k>n$ this component will only contain some fraction of polynomials. The raising and lowering operators of $\mathfrak{A}_n$ act as follows
\begin{equation}
\mathfrak{e}^k_{i,j}: M_n^k\to M_n^{k+1}\,,\quad \mathfrak{f}^k_{i,j}: M_n^k\to M_n^{k-1}\,,
\end{equation}
by adding and removing boxes to/from $\lambda$ at location $(i,j)$. While currently we do not have explicit formulae for these operators, but their existence can be inferred from Theorem \ref{Prop:SchVassHom}.

\subsection{Power-Symmetric Basis}
In order to find the connection with the K-theory of Hilbert schemes of points we need to make a change of variables in Macdonald polynomials $P_\lambda(\boldsymbol{\zeta};q,\hbar)$ as follows
\begin{equation}
x_a =\sum_{i=1}^n \zeta^a_i\,.
\label{eq:PowerSymBasis}
\end{equation}
Once this power-symmetric change of variables is implemented polynomials $P_\lambda$ only depend on the number of boxes of tableau $\lambda$ -- $k$ do not explicitly depend on $n$.

\section{Hilbert Scheme of Points on $\mathbb{C}^2$}\label{Sec:HIlb}
Macdonald polynomials can be identified with classes of fixed points in equivariant K-theory of Hilbert scheme of points on a plane. $\text{Hilb}^n(\mathbb{C}^2)$, where $n$ is the number of points, is defined as a space of ideals $\mathcal{J}$ in $\mathbb{C}[x,y]$ of colength $n$
\begin{equation}
\text{dim}_\Complex \Complex[x,y]/\mathcal{J} = n\,.
\label{eq:idealHilbdim}
\end{equation}
We work equivariantly with respect to the maximal torus $(\Complex^\times)^2$ of $\Complex^2$ with weights $q$ and $\hbar$ acting as $(q,\hbar)p(x,y)=p(qx,\hbar y)$. It is known that the ideals of finite codimention which are fixed under the action of this torus are in one-to-one correspondence with partitions $\lambda$ of $n$. Such ideals are generated by monomials $\{ x^{\lambda_1}, x^{\lambda_2}y, \dots, x^{\lambda_m}y^m, y^{m+1} \}$ where $m$ is the number of columns in tableau $\lambda$.

The localized equivariant K-theory of $\text{Hilb}^n$ can be identified \cite{Bridgeland:,Haiman:} with the space of polynomials
\begin{equation}
K_{q,\hbar}(\text{Hilb}^n) = \Complex[x_1,\dots,x_n]\otimes\Complex[q,t]\,.
\label{eq:KthHilbn} 
\end{equation}
Classes of fixed points $[\mathcal{J}]$ of $(\Complex^\times)^2$ parameterized by $\lambda$ provide a canonical basis in the K-theory. They coincide with Macdonald polynomials of variables $x_1,\dots, x_n$ \eqref{eq:PowerSymBasis}.

\subsection{Elliptic Hall Algebra $\mathfrak{E}$}
Following \cite{Schiffmann:2008aa,Schiffmann:2009} algebra $\mathfrak{E}$ is generated by elements $Z_{n,m},\, n,m\in\mathbb{Z}$ modulo certain relations which we shall not specify here. These generators can be conveniently positioned in integral lattice $\mathbb{Z}^2$ of the coordinate plane where $(n,m)$ will be their $x$ and $y$-coordinates. By normalization $Z_{0,0}=1$ is the identity operator. The commutation relations of $\mathfrak{E}$ suggest that the entire algebra can be generated by four elements $Z_{0,\pm 1}, Z_{\pm 1,0}$. In the given normalization generators on the horizontal axis $Z_{n,0}$ are multiplications by $x^n$, whereas generators on the vertical axis $Z_{0,l}$ for $l>0$ are Macdonald operators written in $x$-basis. In \eqref{eq:SpericalDAHAGen} we have already used these generators to construct $\mathfrak{A}_n$. It is known \cite{Burban:} that $\mathfrak{E}$ acts faithfully on $K_{q,\hbar}(\text{Hilb}^n)$.

This following theorem by Schiffmann and Vasserot enabled the authors to construct (more or less by definition) $\mathfrak{E}$ as a stable limit of spherical DAHA $\mathfrak{A}_\infty$.
\begin{theorem}[\cite{Schiffmann:2008aa}]
There is an explicit surjective algebra homomorphism 
\begin{equation}
\mathfrak{E}\to\mathfrak{A}_n\,.
\end{equation}
\label{Prop:SchVassHom}
\end{theorem}
Using this theorem we can study the connection between modules of $\mathfrak{A}_n$ and $\mathfrak{E}$.

\subsection{Fock Modules of Elliptic Hall Algebra}
Generators $Z_{m,n}$ for which $\frac{m}{n}=s\in\mathbb{Q}$ form a subalgebra of $\mathfrak{E}$ which can be described as $q,\hbar$-deformed Heisenberg algebra with commutation relations
\begin{equation}
[a_m,a_n]=m\frac{1-q^{|m|}}{1-\hbar^{|m|}}\delta_{m,-n}\,,
\label{eq:commutatorsa}
\end{equation}
where $a_n, n\in\mathbb{Z}$ are the corresponding generators of $\mathfrak{E}$ which lie on slope $s$ \cite{Burban:}. The entire $\mathbb{Z}^2$ plane parameterizes by $Z_{m,n}$ can be sliced by lines with all possible rational slopes passing through the origin. Each slope represents itself a Heisenberg algebra. Here we are using slightly different normalization of generators. This normalization is more appropriate in Ding-Iohara algebra which, as we have already mentioned, is isomorphic to $\mathfrak{E}$.

Given the above Heisenberg algebra we can study its Fock space representation $F(a)$ which is constructed in the standard way.

\begin{proposition}
\label{Prop:FockDI}
Let $M_n$ be a highest weight module of $\mathfrak{A}_n$. Then its projective limit $n\to\infty$ is isomorphic to the Fock module $F(a)$ of $\mathfrak{E}$ with evaluation parameter $a$ \eqref{eq:aequivspec}.
\end{proposition}

\begin{proof}
We can map Macdonald polynomials to states in the Fock space representation of the $q,\hbar$-Heisenberg algebra by claiming that
\begin{equation}
x_k = a_{-k}\vert 0\rangle\,,
\label{eq:symmetric}
\end{equation}
where $a_l$ obey \eqref{eq:commutatorsa}.
For a partition of size $k$ and length $n$ we can define $x_\lambda=x_{\lambda_1}\cdots x_{\ell(\lambda)}$ and the correspondingly state in the Fock module $F$ of $\mathfrak{E}$ as $a_{-\lambda_1}\cdots a_{-\ell(\lambda)}$.
Now define a homomorphism $\rho^{n+1}_n: \Lambda_{n+1}\to \Lambda_n$ as
\begin{equation}
(\rho^{n+1}_n f)(\zeta_1,\dots,\zeta_n) = f (\zeta_1,\dots,\zeta_n,0)
\end{equation}
for $f\in\Lambda_{n+1}$. The following projective limit
\begin{equation}
\Lambda := \lim\limits_{\leftarrow n}\Lambda_n
\end{equation}
defines a ring of symmetric functions which is the Fock module of $\mathfrak{E}$, where $\{x_\lambda\}$ form a basis.
Indeed, it can be shown that the inner product on $\Lambda$ defined as 
\begin{equation}
\langle x_\lambda, x_\mu \rangle = \delta_{\lambda,\mu}z_\lambda (q,\hbar)\,,
\end{equation}
where 
\begin{equation}
z_\lambda (q,\hbar)=\prod\limits_{a\geq 1} a^{n_\lambda(a)} n_\lambda(a)! \cdot \prod\limits_{i=1}^{\ell(\lambda)}\frac{1-q^{\lambda_i}}{1-\hbar^{\lambda_i}}\,,
\end{equation}
where $n_\lambda(a)=\#\{i|\lambda_i=a\}$\, which arises from orthogonality of the Macdonald polynomials is in one-to-one correspondence with the following Fock paring
\begin{equation}
\langle 0| a_\lambda a_{-\mu}|0\rangle = \delta_{\lambda,\mu}z_\lambda (q,\hbar)\,.
\end{equation}
The evaluation parameter of the Fock module will be discussed later in Proposition \ref{Prop:Eigenvalues}.
\end{proof}
More details can be found in multiple sources including \cite{2011ntqi.conf..357S,2013arXiv1309.7094S,2013arXiv1301.4912S}.

Finally we can relate the K-theory of Hilbert scheme of points on $\mathbb{C}^2$ with the K-theory of the moduli spaces of quasimaps to $X_n$ \eqref{eq:Kthquasimaps}.

\begin{theorem}
\label{Th:EmbeddingTh}
For $n>k$ there is the following embedding of Hilbert spaces
\begin{align}
\bigoplus\limits_{l=0}^k K_{q,\hbar}(\text{Hilb}^l(\mathbb{C}^2)) &\hookrightarrow \mathcal{H}_n
\label{eq:Embeddingn}\\
[\lambda]&\mapsto \mathsf{V}_{\textbf{q}}\,.\notag
\end{align}
for K-theory vertex function for some fixed point $q$ of maximal torus $T$.
The statement also holds in the limit $n\to\infty$
\begin{equation}
\bigoplus\limits_{l=0}^\infty K_{q,\hbar}(\text{Hilb}^l(\mathbb{C}^2)) \hookrightarrow \mathcal{H}_\infty\,,
\label{eq:Embeddinginf}
\end{equation}
where $\mathcal{H}_\infty$ is defined as a stable limit of $\mathcal{H}_n$ as $n\to\infty$.
\end{theorem}

\begin{proof}
The above maps are constructed by assigning Macdonald polynomials corresponding to fixed points $[\lambda]$ to Macdonald polynomials which are obtained from coefficient vertex functions truncated due to condition \eqref{eq:aequivspec}. Since, according to Proposition \ref{Prop:Truncation}, there is only one coefficient function $\mathsf{V}_{\textbf{q}}$ for which the truncation occurs, this defines a map from $[\lambda]\in\bigoplus\limits_{l=0}^k K_{q,\hbar}(\text{Hilb}^l(\mathbb{C}^2))$ to the vertex function for trivial tautological class in $X_n$.
From Proposition \ref{Prop:FockDI} it follows that inner product is preserved under the embeddings \eqref{eq:Embeddingn, eq:Embeddinginf}, thus they represent isometries of the corresponding Hilbert spaces.
\end{proof}

For future reference we shall denote $\mathcal{M}_1=\bigoplus\limits_{l=0}^\infty K_{q,\hbar}(\text{Hilb}^l(\mathbb{C}^2))$ for rank-one sheaves on $\mathbb{P}^2$ with framing at $\mathbb{P}^1_\infty$. In the next section we shall consider moduli spaces $\mathcal{M}_N$ of rank-$N$ sheaves.

\subsection{Matching Spectra of Macdonald Operators}
Equivariant classes of fixed points $[\lambda]$ in $K_{q,\hbar}(Hilb^k)$ are in one-to-one correspondence with skyscraper sheaves at these points $\mathcal{O}/[\mathcal{J}_\lambda]$. Universal quotient sheaf on Hilb$^k\times\mathbb{C}^2$ gives rise to a degree $k$ tautological bundle $\mathscr{V}\vert_{\mathcal{J}_\lambda}$.
\begin{lemma}
The eigenvalue of the operator of multiplication by the universal bundle corresponding to $\mathscr{V}\vert_{\mathcal{J}_\lambda}$ 
\begin{equation}
\mathscr{U} = \mathscr{W}+(1-\hbar)(1-q)\mathscr{V}\vert_{\mathcal{J}_\lambda}\,,
\end{equation}
where $\mathscr{W}$ is a constant bundle of degree 1, in equivariant K-theory $K_{q,\hbar}(\text{Hilb}^n)$ is given by
\begin{equation}
\mathcal{E}_1(\lambda) = a\left(1-(1-\hbar)(1-q)\sum\limits_{(i,j)\in\lambda}^n \sum_{c=1}^k s_c\right)\,,
\label{eq:EUnivBundleEigen}
\end{equation}
where $s_1,\dots s_k$ are in one-to-one correspondence with the content of tableau $\lambda$ of size $k$ corresponding to class $[\mathcal{J}]$ and are given by
\begin{equation}
s_{i,j} = q^{i-1}\hbar^{j-1}\,,
\end{equation}
for $i,j$ ranging through the arm lengths and leg lengths of $\lambda$. In \eqref{eq:EUnivBundleEigen} $a\in\mathbb{C}^\times$ is the character of $T\mathscr{W}$. 
\end{lemma}

Using the above result one can also calculate the eigenvalues $\mathcal{E}_r(\lambda)$ of the operator of multiplication by the $r$-th anti-symmetric power of the universal bundle $\Lambda^r \mathscr{U}$.


We shall now reproduce formula \eqref{eq:EUnivBundleEigen} from the eigenvalues of tRS operators at locus \eqref{eq:aequivspec} as well as make a more detailed connection between \eqref{eq:KthHilbn} and K-theory of quasimaps into $X$ from before.

\begin{proposition}\label{Prop:Eigenvalues}
The eigenvalues of the $n$-particle tRS model $e_r$ in \eqref{eq:tRSEigenz} and the eigenvalues of multiplication by bundle $\Lambda^r \mathscr{U}$ over $\text{Hilb}^n$ \eqref{eq:EUnivBundleEigen} are in one-to-one correspondence. In particular, for $r=1$
\begin{equation}
\mathcal{E}_1(\lambda)=a \left(\hbar^{n}+\hbar^{n-1}(1-\hbar) e_1\right)\,.
\label{eq:Eenergy}
\end{equation}
\end{proposition}

\begin{proof}
We get 
\begin{equation}
\mathcal{E}_1(\lambda)=a\left(1+(1-\hbar)\sum_{i=1}^k (q^{\lambda_i}-1)\hbar^{i-1}\right)\,.
\end{equation}
The eigenvalue of the first tRS Hamiltonian reads
\begin{align}
a^{-1}e_1 &= \sum\limits_{i=1}^n q^{\lambda_i}\hbar^{i-n} = \hbar^{1-n}\sum\limits_{i=1}^n q^{\lambda_i}\hbar^{1-i}+\hbar^{1-n}\sum\limits_{i=1}^n \hbar^{1-i}-\hbar^{1-n}\sum\limits_{i=1}^n \hbar^{i-1} \cr 
&= \hbar^{1-n}\sum\limits_{i=1}^n (q^{\lambda_i}-1)\hbar^{i-1}+\hbar^{1-n}\frac{1-\hbar^{n}}{1-\hbar}\,.
\end{align}
Combining the last two formulae we get \eqref{eq:Eenergy}.
\end{proof}

Analogously we can prove the correspondence between the eigenvalues of higher Hamiltonians $e_r$ and $\mathcal{E}_r(\lambda)$ (see \cite{Feigin:2009ab} for details). 

Since Proposition \ref{Prop:Eigenvalues} works for any finite $n$ we can study the limit $n\to\infty$. The following statement follows directly from the above formulae
\begin{proposition}\label{Prop:StableLimitEigen}
Assuming that $|\hbar|<1$ we have
\begin{equation}
\mathcal{E}_1(\lambda)= \lim\limits_{n\to\infty}\left(\hbar^{n-1}(1-\hbar) e_1\right)\,.
\label{eq:EenergyLim}
\end{equation}
\end{proposition}

As $e_r$ are eigenvalues of Macdonald operators which are used in the the construction of modules $M_n$ of $\mathfrak{A}_n$ we conclude that character $a$ plays a role of the evaluation parameter in the Fock module of $\mathfrak{E}$ in Proposition \ref{Prop:FockDI}.

\vspace{2mm}
\paragraph{\textbf{Example.}}
Let us illustrate Proposition \ref{Prop:StableLimitEigen} for $k=2$. There are two possible Young tableaux with two boxes. Assuming $a=1$ we get
\begin{itemize}
\item $\lambda = {\tiny\yng(2)}$. For some large $n>2$ we view ${\tiny\yng(2)}$ as tableau with $n$ columns: $(1,1,0,\dots, 0)$. Then the eigenvalue for tRS operator $T_1$ is given by the sum of equivariant parameters at locus \eqref{eq:aequivspec}
\begin{equation}
e_1=q^2 \hbar^{1-n}+\hbar^{2-n}+\dots+\hbar+1\,.
\end{equation}
Using \eqref{eq:EenergyLim} and talking $n\to\infty$ we get $\mathcal{E}_1 = q^2+\hbar-q\hbar$.

\item $\lambda = {\tiny\yng(1,1)}$. Analogously we have tableau with $n$ columns most of which are zeroes $(2,0,\dots, 0)$. tRS eigenvalue $e_1=q \hbar^{1-n}+q\hbar^{2-n}+\hbar^{3-n}+\dots+1$ and $\mathcal{E}_1 = q+\hbar-q \hbar^2$.
\end{itemize}

\vspace{2mm}
\paragraph{\textbf{Remark.}}
We have proved that the eigenvalue of the multiplication operator by the $r$-th power of the universal bundle over Hilb$^k$ can be evaluated as a stable limit of the eigenvalue of Macdonald operator $T_r$ \eqref{eq:tRSRelationsEl}. This statement illustrates recently demonstrated correspondence between spectrum of the tRS model with large number of particles and hydrodynamical Benjamin-Ono model \cite{Koroteev:2018a}.

\subsection{The Correspondence}
We can summarize the correspondence between geometry of quiver varieties and Hilbert schemes of points in the following table where we already assume the stable limit in $n$.

\begin{table}[!h]
\begin{center}
\renewcommand\arraystretch{1.2}
\begin{tabular}{|c|c|}
\hline
$K_T(\textbf{QM}(\mathbb{P}^1,X))$ & $K_{q,\hbar}(\text{Hilb}(\mathbb{C}^2))$ \\ 
\hline\hline K\"ahler/quantum parameters of $X$ $z_1,z_2\dots$  & Ring generators $x_1,x_2,\dots$ \\ 
\hline Vertex function $\mathsf{V}_{\textbf{q}}$ at locus \eqref{eq:aequivspec} & Classes of $(\mathbb{C}^\times)^2$ fixed points $[\mathcal{J}]$ \\
\hline $\mathbb{C}^\times_q$ acting on base curve & $\mathbb{C}^\times_q$ acting on $\mathbb{C}\subset\mathbb{C}^2$ \\
\hline $\mathbb{C}^\times_\hbar$ acting on cotangent fibers of $X$ & $\mathbb{C}^\times_\hbar$ acting on another $\mathbb{C}\subset\mathbb{C}^2$ \\
\hline Eigenvalues $e_r$ of tRS operators $T_r$ &  Chern polynomials $\mathcal{E}_r$ of $\Lambda^r\mathcal{U}$ \\
\hline
\end{tabular}
\vspace{2mm}
\renewcommand\arraystretch{1}
\caption{The correspondence between K-theories of quasimaps to $X_n$ and Hilb.}
\label{Tab:Corresp}
\end{center}
\end{table}

\vspace{2mm}

Another way to formulate the duality is the following. The \textit{quantum} K-theory of $X$ (which is by construction $K_T(\textbf{QM}(\mathbb{P}^1,X))$) is related to \textit{classical} equivariant K-theory of the Hilbert scheme of  points in $\mathbb{C}^2$. A natural question arises: what needs to be modified in the left column of the above table to obtain \textit{quantum} version of the space on the right? The answer will be formulated later in \secref{Sec:QuantumMult}

\subsection{Remarks}
Notice that in the left column of \tabref{Tab:Corresp} equivariant parameters $q$ and $\hbar$ play completely different roles -- the former scales the base curve, while the latter corresponds to $\mathbb{C}^\times$ action on the cotangent directions of $X_n$. In the right column they are on completely equal footing and can be interchanged.

The above duality works only when equivariant parameters $a_1,a_2,\dots, a_n$ in $X_n$ obey \eqref{eq:aequivspec}. It was discussed in the literature on integrable systems and gauge theories \cite{Nekrasov:2009rc} and more recently in \cite{Sciarappa:2016} that locus \eqref{eq:aequivspec} should be interpreted as a set of quantization conditions which allow for discrete spectrum of the tRS model. Thus, due to the above duality, we can observe a symmetry in the tRS spectrum which interchanges (exponential of) Planck constant $q$ with coupling constant $\hbar$ of the tRS model.

\section{Moduli Space of Sheaves of Rank $N$ on $\mathbb{C}^2$}\label{Sec:RankNSheaves}
Let $\mathcal{M}_{N}$ denote the moduli space of rank $N$ of torsion free sheaves $\mathcal{F}$ of rank $N$ on $\mathbb{P}^2$ with framing at infinity: $\phi: \mathcal{F}|_\infty \simeq \mathcal{O}^{\oplus N}_\infty$. The framing condition forces the first Chern class to vanish, however, the second Chern class can range over the non-positive integers $c_2(\mathcal{F})=-k\cdot[\text{pt}]$. Therefore the moduli space can be represented as a direct sum of disconnected components of all degrees 
\begin{equation}
\mathcal{M}_{N} = \bigsqcup_k \mathcal{M}_{N,k}\,.
\end{equation}
More details can be found in multiple sources, i.e \cite{Feigin:2009aa, Negut:2012aa}.

There is an action of maximal torus $T_N:=\mathbb{C}_q^\times\times\mathbb{C}_\hbar^\times\times\mathbb{T}(GL(N))$ on each component $\mathcal{M}_{N,k}$. The first two $\mathbb{C}^\times$ factors act on $\mathbb{P}^2$, while the rest acts on framing $\phi$ with equivariant parameters $\mathrm{a}_1,\dots \mathrm{a}_N$. We shall be studying $T_N$-equivariant K-theory of $\mathcal{M}_{N,k}$ which is formed by virtual equivariant vector bundles on $\mathcal{M}_{N,k}$. The space $K_{q,\hbar}(\mathcal{M}_{N,k})$ is a module over $\mathbb{C}[q^{\pm 1},\hbar^{\pm 1},\mathrm{a}_1^{\pm 1},\dots \mathrm{a}_N^{\pm 1}]$. 

$T_N$-fixed points are isolated and indexed by $N$-tuples of partitions $\vec{\lambda}=\{\lambda_1,\dots,\lambda_N\}$ with the sum of sizes of all partitions $\sum_{l=1}^N |\lambda_l| = k$
\begin{equation}
\mathcal{F}_{\vec{\lambda}} = \mathcal{J}_{\lambda_1}\oplus\cdots\oplus\mathcal{J}_{\lambda_N}\,.
\label{eq:SkysraperSheaf}
\end{equation}
Here $\mathcal{J}_{\lambda_l}$ are monoidal ideals in $\mathbb{C}[x,y]$ same as in Hilb$^k$ \eqref{eq:idealHilbdim}.
Analogously to Hilb$^k$, equivariant classes $[\vec{\lambda}]$, which can be thought of skyscraper sheaves of $\mathcal{F}_{\vec{\lambda}}$, form a basis in localized K-theory of $\mathcal{M}_{N,k}$.

The space $\mathcal{M}_{N}$ can also be described as the moduli spaces of stable representations of ADHM quiver.
For $N=1$ this quiver variety describes Hilbert schemes of $k$ points on $\mathbb{C}^2$. The two descriptions in terms of sheaves and the ADHM are equivalent \cite{Nakajima:2014aa}.

\subsection{Asymptotic partitions}
We shall modify the techniques of the previous section which studied $\text{Hilb}^k$ to study K-theory of $\mathcal{M}_N$. The main idea was developed in physics paper by the author and A. Sciarappa \cite{Koroteev:2016}, here we shall make all the statements mathematically precise.

The starting point is the moduli space of quasimaps to cotangent bundle to complete flag variety in $\mathbb{C}^{nN}$
\begin{equation}
\mathcal{H}_{nN}=K_{T_N}(\textbf{QM}(\mathbb{P}^1,X_{nN}))\,,
\label{eq:KthquasimapsnN}
\end{equation}
where parameter $n$ will later be sent to infinity. The main idea is to study the above space at a locus similar to \eqref{eq:aequivspec} with the following modification.  

An element in $K_{q,\hbar}(\mathcal{M}_{N,k})$ is parameterized by a tuple of tableaux $\{\lambda_1,\dots,\lambda_N\}$ such that the total number of boxes $\sum_{l=1}^N |\lambda_l|=k$. Let $a_1,\dots,a_{nN}$ be equivariant parameters of the maximal torus of $GL(nN)$ acting on the framing node $\textbf{w}_{nN}$ of the quiver.
Then the sought connection between $\mathcal{H}_{nN}$ and the moduli space of sheaves of rank $N$ is achieved
by, first, splitting the equivariant parameters into $N$ groups $\{a_i^{(l)}\},\,i=1,\dots,n,\,l=1,\dots, N$ and, second, setting their values to the following
\begin{equation}
\frac{a^{(l)}_{i+1}}{a^{(l)}_{i}} = q^{\ell^{(l)}_i}\, \hbar^{i-n}\,, \quad \ell_i^{(l)} =\lambda^{(l)}_i-\lambda^{(l)}_{i-1}\,,
\label{eq:LocusN}
\end{equation}
for some generic $\mathbb{C}^\times$-valued constants $\mathrm{a}_l$.

Equivalently 
\begin{equation}
a_\alpha = \widetilde{\mathrm{a}_\alpha}\, q^{\Lambda_\alpha}\, \hbar^{\alpha-nN}\,,\qquad \alpha = 1,\dots, nN\,,
\end{equation}
where  
\begin{equation}
\{\widetilde{\mathrm{a}_\alpha}\} = \{\mathrm{a}_1,\dots \mathrm{a}_1,\mathrm{a}_2,\dots,\mathrm{a}_2,\dots,\mathrm{a}_N,\dots,\mathrm{a}_N\}
\end{equation}
and Young tableau $\Lambda$ is shown in \figref{Fig:AsymptPart} and is constructed as follows \cite{Schiffmann:2009}\footnote{Schiffmann and Vasserot used this asymptotic partition to show that there is a triangular decomposition in Heisenberg algebras inside $\mathfrak{E}$.
(see Proposition 4.8. in \textit{loc. cit.})}. We define
\begin{equation}
\Lambda=\lambda_1\circledast\ldots\circledast\lambda_N
\end{equation}
by placing each of tableaux $\lambda_l,\, l=1,\dots, N$ into corners of a `staircase' tableau with step size equal to $n$. We always assume that $n$ is parametrically larger than $k$, so the number of columns as well as heights of columns of any $\lambda_l$ is always smaller than $n$.
\begin{figure}[!h]
\centerline{
\begin{picture}(200,200)
\put(20,170){{$\lambda_1$}}
\put(75,115){{$\lambda_2$}}
\put(185,17){{$\lambda_N$}}
\put(57,157){{\small{$(n,(N-1)n)$}}}
\put(112,102){{\small{$(2n,(N-2)n)$}}}
\put(167,57){{\small{$((N-1)n,n)$}}}
\multiput(0,155)(5,0){5}{\line(1,0){3}}
\put(0,0){\line(0,1){180}}
\put(0,0){\line(1,0){200}}
\put(0,180){\line(1,0){7}}
\put(7,180){\line(0,-1){5}}
\put(7,175){\line(1,0){3}}
\put(10,175){\line(0,-1){5}}
\put(10,170){\line(1,0){5}}
\put(15,170){\line(0,-1){10}}
\put(15,160){\line(1,0){10}}
\put(25,160){\line(0,-1){5}}
\put(25,155){\line(1,0){30}}
\multiput(55,100)(0,5){4}{\line(0,1){3}}
\multiput(55,100)(5,0){5}{\line(1,0){3}}
\put(55,155){\line(0,-1){35}}
\put(55,120){\line(1,0){5}}
\put(60,120){\line(0,-1){5}}
\put(60,115){\line(1,0){5}}
\put(65,115){\line(0,-1){5}}
\put(65,110){\line(1,0){10}}
\put(75,110){\line(0,-1){5}}
\multiput(165,0)(0,5){4}{\line(0,1){3}}
\put(75,105){\line(1,0){5}}
\put(80,105){\line(0,-1){5}}
\put(80,100){\line(1,0){30}}
\multiput(112,78)(5,-5){5}{\circle*{1}}
\put(110,100){\line(0,-1){20}}
\put(135,55){\line(1,0){30}}
\put(165,55){\line(0,-1){35}}
\put(165,20){\line(1,0){5}}
\put(170,20){\line(0,-1){5}}
\put(170,15){\line(1,0){10}}
\put(180,15){\line(0,-1){5}}
\put(180,10){\line(1,0){5}}
\put(185,10){\line(0,-1){5}}
\put(185,5){\line(1,0){15}}
\put(200,5){\line(0,-1){5}}
\end{picture}}
\caption{The asymptotic partition $\Lambda$}
\label{Fig:AsymptPart}
\end{figure}
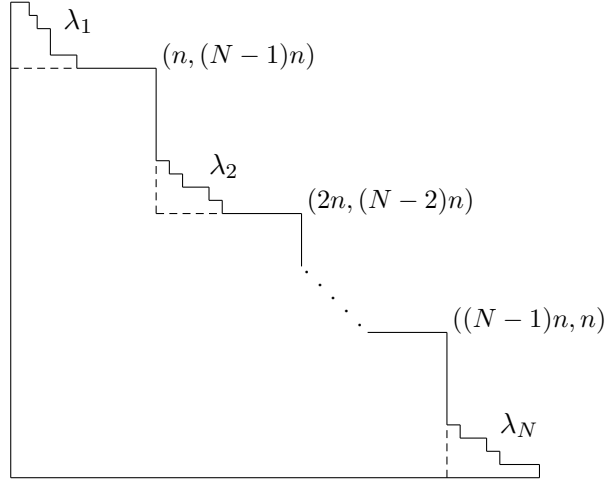

Asymptotic partitions $\Lambda$ describe equivariant K-theory classes of fixed points and, similarly to the previous section, provide a basis in the K-theory of $\mathcal{M}_{N,k}$.

\subsection{Highest weight modules of $\mathfrak{A}_n$ and Tensor Product of Fock Modules of $\mathfrak{E}$}
Consider highest weight modules of $\mathfrak{A}_{nN}$ which are parameterized by generalized Macdonald polynomials
\begin{equation}
\label{eq:GeneralizedMacs}
M_{nM}=\left\{P_{\Lambda}(\boldsymbol{\zeta}_1,\dots,\boldsymbol{\zeta}_N | q,\hbar)\,\Big| \ell(\lambda_l)< n,\,l=1,\dots,N\right\}\,,
\end{equation}
which depend on $N$-tuple of variables $\boldsymbol{\zeta}_l = \zeta^{(l)}_1,\dots,\zeta^{(l)}_{\ell(\lambda_l)},\,l=1,\dots, N$. These polynomials simultaneously obey $N$ decoupled tRS difference relations
\begin{equation}
T^{(l)}_r P_{\Lambda} = s_r(a^{(l)}_1,\dots,a^{(l)}_n) P_{\Lambda}\,,\qquad l=1,\dots, N\,.
\end{equation}
Therefore module $M_{nN}$ decomposes in the tensor product of $N$ modules \eqref{eq:Mnmodule}
\begin{equation}
M_{nN} = M_{n}\otimes\cdots\otimes M_{n}\,.
\end{equation}
Combining this observation with Proposition \ref{Prop:FockDI} we arrive to
\begin{proposition}\label{Prop:StabLimTen}
Let $M_{nN}$ be a highest weight module of $\mathfrak{A}_{nN}$. Then its projective limit $n\to\infty$ is isomorphic to the tensor product of $N$ modules with evaluation parameters $\mathrm{a}_1,\dots,\mathrm{a}_N$
\begin{equation}
\mathrm{T}_N= F(\mathrm{a}_1)\otimes\cdots\otimes F(\mathrm{a}_N)
\label{eq:TensorFockModule}
\end{equation}
of $\mathfrak{E}$.
\end{proposition}

Thus states in the highest weight module $M_{nN}$ can be matched to vectors in the tensor Fock module and the generalized Macdonald polynomials are mapped onto ideals \eqref{eq:SkysraperSheaf} in the K-ring of $\mathcal{M}_{N}$.

\subsection{Matching Spectra of Macdonald Operators}
Analogously with the Hilbert scheme on $\mathbb{C}^2$ we can study K-theory of $\mathcal{M}_{N,k}$, which is generated by the classes of fixed points $[\vec{\lambda}]$ of maximal torus $T_N$. 

\begin{lemma}
The eigenvalues of the operator of multiplication by the universal bundle 
\begin{equation}
\mathscr{U}=\mathscr{W}+(1-q)(1-\hbar)\mathscr{V}\vert_{\mathcal{J}_{\vec{\lambda}}}
\end{equation}
over $\mathcal{M}_{N,k}$, where $\mathscr{W}$ is a constant bundle of degree $N$ and tautological bundle $\mathscr{V}\vert_{\mathcal{J}_{\vec{\lambda}}}$ arise from the universal quotient sheaf on $\mathcal{M}_{N,k}\times\mathbb{C}^2$ in K-theory of $\mathcal{M}_{N,k}$ is given by 
\begin{align}
\mathcal{E}_1(\Lambda) &=  \sum_{c=1}^N \mathrm{a}_a-(1-\hbar)(1-q)\sum\limits_{l=1}^N \sum\limits_{(i,j)\in\lambda^{(l)}}\sum_{d=1}^{k_c} s^{(l)}_d
\cr
&=\sum_{c=1}^N \mathrm{a}_a-(1-\hbar)(1-q)\, \sum\limits_{(i,j)\in\Lambda}\sum_{c=1}^k \mathrm{s}^{(l)}_c\,,
\label{eq:EUnivBundleEigenN}
\end{align}
where $\sum_{l=1}^N |\lambda_l|=\sum_{l=1}^N k_l=k$ and 
$s^{(l)}_1,\dots,s^{(l)}_{k_l}$ are in one-to-one correspondence with the content of tableaux $\lambda^{(l)}$ of size $k_l$ \begin{equation}
s^{(l)}_{i,j} = q^{i-1}\hbar^{j-1}\,,
\end{equation}
for $i,j$ ranging through the arm length and leg lengths of tableaux $\lambda^{(l)}$. Variables $\mathrm{s}^{(l)}_1,\dots,\mathrm{s}^{(l)}_{k}$ are in one-to-one correspondence with the content of the asymptotic partition $\Lambda$, where $\mathrm{s}_{i,j}$ ranges over the content of $\Lambda$. In \eqref{eq:EUnivBundleEigenN} $e_1(\mathrm{a}_1,\dots,\mathrm{a}_N)$ are characters of $T\mathscr{W}$.
\end{lemma}

Using the above result we can compute eigenvalues $\mathcal{E}_r(\Lambda)$ of the operator of multiplication by the $r$-th skew-power of the universal bundle $\Lambda^r\mathscr{U}$.


\vspace{2mm}

Now we can compare these eigenvalues with the tRS eigenvalues.

\begin{proposition}
\label{Prop:EigenvaluesN}
The eigenvalues of the $nN$-particle tRS model $e_r$ and the eigenvalues of multiplication by $r$-th skew symmetric power of the universal bundle $\Lambda^r\mathscr{U}$ over $\mathcal{M}_{N,k}$ are in one-to-one correspondence. In particular
\begin{equation}
\mathcal{E}_1(\Lambda)= \hbar^{n}\sum_{l=1}^N\mathrm{a}_l+\hbar^{n-1}(1-\hbar) e_1\,.
\label{eq:EenergyN}
\end{equation}
\end{proposition}

\begin{proof}
Let us put $r=1$ in \eqref{eq:EUnivBundleEigenN}. Then we get 
\begin{equation}
\mathcal{E}_1(\Lambda)=\hbar^{n}\sum_{l=1}^N\mathrm{a}_l+\hbar^{n}(1-\hbar)e_1=
\sum_{l=1}^N\mathrm{a}_l+(1-\hbar)\sum_{l=1}^N\sum_{a=1}^n\mathrm{a}_l (q^{\lambda_a^{(l)}}-1)\hbar^{a-1}
\,.
\end{equation}
The eigenvalue of the first tRS Hamiltonian evaluated at \eqref{eq:LocusN} reads
\begin{equation}
e_1 = \sum_{l=1}^N\sum_{a=1}^n\mathrm{a}_l q^{\lambda_a^{(l)}}\hbar^{n-a}
 = \hbar^{n-1}\sum\limits_{i=1}^n \mathrm{a}_l(q^{\lambda_i}-1)\hbar^{a-1}+\hbar^{n-1}\frac{1-\hbar^{n}}{1-\hbar}\sum_{a=1}^n\mathrm{a}_l\,.
\end{equation}
Combining the last two formulae we get \eqref{eq:EenergyN}.
\end{proof}

We can directly obtain the following statement which coincides with Proposition \ref{Prop:StableLimitEigen} up to replacement of $\lambda$ with $\Lambda$.
\begin{proposition}\label{Prop:StableLimitEigenN}
Assuming that $|\hbar|<1$ we have
\begin{equation}
\mathcal{E}_1(\Lambda)= \lim\limits_{n\to\infty}\left(\hbar^{n-1}(1-\hbar) e_1\right)\,.
\label{eq:EenergyN}
\end{equation}
\end{proposition}

We leave to the reader the proof for of this proposition for higher Hamiltonians $T_r$.

\subsection{MacMahon Modules of $\mathfrak{E}$}
It was shown in \cite{Feigin:2010,Feigin:2012aa} (see also \cite{Zenkevich:2017aa}) that once certain conditions, called resonances, are met on the evaluation parameters $\mathrm{a}_1,\dots,\mathrm{a}_N$ one can find a one-to-one correspondence between the tensor product module $\mathrm{T}_N$ \eqref{eq:TensorFockModule} and MacMahon module $\mathrm{M}$ of $\mathfrak{E}$. Let us discuss this embedding from the point of view of Proposition \ref{Prop:StabLimTen}.

Starting from \eqref{eq:LocusN} we additionally impose
\begin{equation}
\mathrm{a}_l=a (q\hbar)^{1-l}\,,
\end{equation}
where $a\in\mathbb{C}^\times$ is an overall scaling parameter, to get
\begin{equation}
a^{(l)}_i = a\,q^{\lambda^{(l)}_i-l+1}\, \hbar^{n-i-l+1}=a\, t_1^{\lambda^{(l)}_i}t_2^{n-i}t_3^{l-1}\,,
\label{eq:LocusNMacMaheval}
\end{equation}
where 
\begin{equation}
t_1=q\,,\qquad t_2=\hbar\,,\qquad t_3=q^{-1}\hbar^{-1}\,,
\label{eqt123change}
\end{equation}
such that $t_1 t_2 t_3=1$. We can see that \eqref{eq:LocusNMacMaheval} provides the data for the content of plane partition $\Pi$
which is obtained from the data of integer partition $\Lambda$ from \figref{Fig:AsymptPart}. This leads us to the following

\begin{proposition}\label{Prop:StabLimTenMac}
Let $M_{nN}$ be a highest weight module of $\mathfrak{A}_n$ whose weights are subject to resonance condition \eqref{eq:LocusNMacMaheval}. Then its projective limit $n\to\infty$ is isomorphic to MacMahon module $\mathrm{M}(a)$ of $\mathfrak{E}$.
\end{proposition}

It was then argued in \cite{Zenkevich:2017aa} that eigenvalues of Cartan generators of $\mathfrak{E}$ are given by \textit{triple} Macdonald polynomials, which can be obtained from generalized Macdonald polynomials \eqref{eq:GeneralizedMacs} after imposing resonance conditions \eqref{eq:LocusNMacMaheval}. This eigenvalue problem has a natural interpretation in terms of the ADHM construction.

\subsection{The Correspondence}
Analogously to Theorem \ref{Th:EmbeddingTh} we can formulate the following theorem for the moduli spaces of rank-$N$ sheaves.
\begin{theorem}
\label{Th:EmbeddingSheavesN}
For $nN>k$ there is the following embedding of Hilbert spaces
\begin{align}
\bigoplus\limits_{l=0}^k K_{q,\hbar}(\mathcal{M}_{N,k}) &\hookrightarrow \mathcal{H}_{nN}
\label{eq:Embeddingn}\\
[\Lambda]&\mapsto \mathsf{V}_{\textbf{q}}\,.\notag
\end{align}
for K-theory vertex function for some fixed point $q$ of maximal torus $T$.
The statement also holds in the limit $n\to\infty$
\begin{equation}
\bigoplus\limits_{l=0}^\infty K_{q,\hbar}(\mathcal{M}_{N,k}) \hookrightarrow \mathcal{H}_\infty\,,
\label{eq:Embeddinginf}
\end{equation}
where $\mathcal{H}_\infty$ is defined as a stable limit of $\mathcal{H}_{nN}$ as $n\to\infty$.
\end{theorem}

\section{Quantum Multiplication in $K_{\textsf{T}}(\mathcal{M}_N)$}\label{Sec:QuantumMult}
In previous two sections we have demonstrated (see Propositions \ref{Prop:Eigenvalues} and \ref{Prop:EigenvaluesN}) that spectra of the trigonometric Ruijsenaars-Schneider Hamiltonians are in one-to-one correspondence with classical multiplication by universal bundles in the equivariant K-theory of the ADHM moduli space. The correspondence is summarized in \tabref{Tab:Corresp}. We shall now discuss \textit{quantum deformation} of the ADHM moduli space in connection with \textit{elliptic} Ruijsenaars-Schneider (eRS) integrable system with large number of particles. Some of the material in this section will be speculative as the corresponding mathematical tools have not yet been fully developed yet. We hope to gain a substantial understanding of these methods in the near future.

This section is also motivated in part by the correspondence between quantum geometry of $\mathcal{M}_N$ and Donaldson-Thomas/Gromov-Witten theory on $\mathbb{P}^1\times\mathbb{C}^2$ (see \cite{Okounkov:2004a,Okounkov:2005a} as well as \cite{Maulik:2008b,Maulik:2008a,Maulik:2008} for generalizations). To the best of our knowledge, all results are obtained in cohomology. One can identify the quantum multiplication operator in $H^\bullet(\text{Hilb})$ with the the elliptic Calogero-Moser Hamiltonians written in Fock representation. The quantum multiplication operator has an expansion in terms of curve counting parameter, let us call it $\mathfrak{p}$, which acts with $\mathbb{C}^\times$ rescalings on the base curve. This expansion was shown to match the expansion of the Weierstrass $\mathcal{P}(x_i-x_j|\mathfrak{p})$ function, which serves as potential in the Calogero-Moser model, in the elliptic parameter \cite{Okounkov:2004a}. 

In seminal paper \cite{Negut:2011aa} (see also \cite{Braverman:aa,Braverman:ab,Braverman:2012aa}) the elliptic Calogero-Moser system was understood both from the viewpoints of both representation theory and algebraic-geometry. The eigenfunctions of its Hamiltonians were proven to be equivariant integrals of certain characteristic classes over the \textit{affine Laumon spaces}. Recently Nekrasov showed \cite{Nekrasov:2017ab,Nekrasov:2017aa} that a similar equivariant integral with full maximal torus\footnote{To get a stationary eigenvalue problem one needs to turn off one of the equivariant parameters on either complex line in $\mathbb{C}^2$.} solves the corresponding \textit{non-stationary} Schr\"odinger equation. 

Once the elliptic Calogero-Moser system is downgraded to the trigonometric one the results of \textit{loc. cit.} apply after reducing affine Laumon space $\mathcal{L}^{\text{aff}}_{\textbf{d}}$ to the its finite version $\mathcal{L}_{\textbf{d}}$. Here vector $\textbf{d}=(d_1,\dots d_s)$ shows degrees of parabolic sheaves which are used in the construction of the Laumon space. For the purposes of our presentation the number of components in $\textbf{d}$ will always be equal to the rank of gauge group of the supersymmetric theory which is used in the construction. In physics language the spectrum of the elliptic Calogero-Moser model is described by instanton counting in $\mathcal{N}=2^*$ gauge theory in the presence of a monodromy defect of maximal Levi type \cite{Alday:2010vg,Nawata:2014nca}.

The sought generalization of the above results to quantum K-theory should be formulated in terms of the \textit{relativistic} generalization of the Calogero-Moser system -- the elliptic Ruijsenaars-Schneider (eRS) model.
Physically we will be studying five-dimensional $\mathcal{N}=1^*$ gauge theory with defect of maximal Levi type 
\cite{Bullimore:2015fr,Koroteev:2018a}.

\subsection{Elliptic Ruijsenaars-Schneider Model}
The Hamiltonians of the elliptic RS model can be easily obtained from trigonometric RS Hamiltonians \eqref{eq:tRSRelationsEl} by replacing rational functions with elliptic theta-functions of the first kind
\begin{equation}
E_r(\boldsymbol{\zeta})=\sum_{\substack{\mathcal{I}\subset\{1,\dots,n\} \\ |\mathcal{I}|=r}}\prod_{\substack{i\in\mathcal{I} \\ j\notin\mathcal{I}}}\frac{\theta_1(\hbar\zeta_i/\zeta_j|\mathfrak{p})}{\theta_1(\hbar\zeta_i/\zeta_j|\mathfrak{p})}\prod\limits_{i\in\mathcal{I}}p_k \,,
\label{eq:eRSRelationsEl}
\end{equation}
where $\mathfrak{p}\in\mathbb{C}^\times$ is the new parameter which characterizes the elliptic deformation away from the trigonometric locus, where $\mathfrak{p}=1$ and we get \eqref{eq:tRSRelationsEl} back.

As in the trigonometric case we shall be interested in the eigenvalues and eigenfunctions of these operators
\begin{equation}
E_r(\boldsymbol{\zeta}) \mathcal{Z} = \mathscr{E}_r\mathcal{Z}\,,\qquad r=1,\dots,n\,.
\label{eq:eRSequation}
\end{equation}

As a direct generalization of the results of \cite{Negut:2011aa} to K-theory lead to the following
\begin{conjecture}\label{Conj:eRSEigenproblem}
The solution of \eqref{eq:eRSequation} is given by the K-theoretic holomorphic equivariant Euler characteristic of the affine Laumon space 
\begin{equation}
\mathcal{Z} = \sum_{\textbf{d}} \vec{\mathfrak{q}}^{\textbf{d}} \int\limits_{\mathcal{L}_{\textbf{d}}} 1\,,
\label{eq:equivaraintKthLaumon}
\end{equation}
where $\vec{\mathfrak{q}}=(\mathfrak{q}_1,\dots,\mathfrak{q}_n)$ is a string of $\mathbb{C}^\times$-valued coordinates on the maximal torus of $\mathcal{L}^{\text{aff}}_{\textbf{d}}$.
The eigenvalues $\mathscr{E}_r$ are equivariant Chern characters of bundles $\Lambda^r \mathscr{W}$, where $\mathscr{W}$ is the constant bundle of the corresponding ADHM space. In other words they have the following form
\begin{equation}
\mathscr{E}_r = e_r + \sum_{l=1}^\infty \mathfrak{p}^l \mathcal{E}^{(l)}_r \,,
\label{eq:eRSEnergies}
\end{equation}
where $e_r$ are symmetric functions of the equivariant parameters $a_1,\dots, a_N$.
\end{conjecture}

The K-theoretic equivariant pushforward \eqref{eq:equivaraintKthLaumon}, as in say \eqref{eq:vertexQKgen}, has not been defined yet for non-hyperK\"ahler spaces. We hope that this will be done in the near future which will enable us to prove the Conjecture.

The affine Laumon space provides a generalization for the ADHM moduli space, in fact the former can be obtained by a certain quotient of the latter by $\mathbb{Z}_n$. The so-called chainsaw quiver provides all necessary data for equivariant localization computations on $\mathcal{L}^{\text{aff}}_{\textbf{d}}$. The fixed points of the maximal torus of $\mathcal{L}^{\text{aff}}_{\textbf{d}}$ are parameterized by an $n$-tuple of Young tableaux $\vec{\mu}=(\mu_1,\dots,\mu_n)$. The integrals in \eqref{eq:equivaraintKthLaumon} can be computed using localization and the resulting expression is an infinite sum over all sectors labelled by $k_l(\vec{\mu})$
\begin{equation}
\mathcal{Z} = \sum_{\vec{\mu}} \prod_{l=1}^n \mathfrak{q}_l^{k_l(\vec{\mu})}\,z_{\vec{\mu}}(\vec{a},\hbar, q)\,.
\label{eq:ZPartFuncRamFull}
\end{equation}

In the limit when the parabolic structure is removed \eqref{eq:equivaraintKthLaumon} is expected to reproduced the well known Euler characteristic of $\mathcal{M}_N$ (Nekrasov instanton partition function) Thus we can impose the following
\begin{equation}
\mathfrak{p}=\mathfrak{q}_1\cdot\dots\cdot\mathfrak{q}_n\,,
\label{eq:ProdInst}
\end{equation}
where $\mathfrak{p}$ counts the degrees of sheaves in the standard ADHM localization computation. We can use the above equation in order to eliminate one of $\mathfrak{q}_i$ variables. In \cite{Bullimore:2015fr} the Conjecture \ref{Conj:eRSEigenproblem} was verified in several orders in $\mathfrak{p}$. 

Using the above observations we can think of \eqref{eq:ZPartFuncRamFull} as a Laurent series in $\mathfrak{q}_1,\dots,\mathfrak{q}_{n-1}$. Indeed, if we express $\mathfrak{q}_n$ in using \eqref{eq:ProdInst} we shall get a series in all positive and negative powers of $\mathfrak{q}_1,\dots,\mathfrak{q}_{n-1}$ and in nonnegative powers of $\mathfrak{p}$. In the limit $\mathfrak{p}\to 0$ (or constant term in $\mathfrak{p}$ in the above Laurent series) formula \eqref{eq:ZPartFuncRamFull} describes holomorphic Euler characteristic of the finite\footnote{non affine} Laumon space, and, as we have proven above, the resulting function is the eigenfunction of tRS Hamiltonians. It is easy to see that at $\mathfrak{p}=0$ we get a Taylor series is $\mathfrak{q}_i$ parameters.  The chainsaw quiver is replaced by a handsaw quiver for the Laumon space (see e.g. \cite{Finkelberg:2010}). 

We can formulate a corollary form Conjecture \ref{Conj:eRSEigenproblem}
\begin{corollary}
The equivariant K-theoretic Euler characteristic of the Laumon space coincides with the vertex function coefficient \eqref{eq:V1pdef} 
for trivial class in $K_T(X_n)$.
\end{corollary}

\begin{proof}
To prove one computes the equivariant character of the universal bundle over the Laumon space in terms of variables $\mathfrak{q}_1,\dots,\mathfrak{q}_{n-1}$ which are then identified with quantum parameters $z_1,\dots, z_n$ in \eqref{eq:Vtauz}.
\end{proof}

According to \secref{Sec:qKZtRS} once condition \eqref{eq:aequivspec} is satisfied the coefficient vertex function of $X_n$ truncates to the corresponding Macdonald polynomial. Based on the above discussion we conclude that generating function \eqref{eq:ZPartFuncRamFull} after imposing \eqref{eq:aequivspec} will become a Laurent series plus this Macdonald polynomial. Also the Laurent part is uniquely fixed by this condition. Therefore functions $\mathcal{Z}$ under \eqref{eq:aequivspec} for all possible Young tableau $\lambda$ span Hilbert space which is isomorphic to $\mathcal{H}_n$, but has a different scalar product due to the presence of additional series in $\mathfrak{p}$. 

\subsection{eRS Eigenvalues}
One can perform the localization computation to compute \eqref{eq:eRSEnergies} (see \cite{Bullimore:2015fr,Koroteev:2018a}). The first several terms for the eigenvalues of $E_1$ look as follows
\begin{equation}
\mathscr{E}_1 = \sum_{i=1}^n a_i -\mathfrak{p}(1-\hbar)(q-\hbar^{-1})q^{-1}\hbar^{n}\sum_{i=1}^n a_i\prod_{\substack{j=1 \\ j\neq i}}^n \frac{(a_i-\hbar^{-1}a_j)(\hbar a_i-q a_j)}{(a_i-a_j)(a_i-qa_j)}+o(\mathfrak{p}^2)\,.
\label{eq:E1eigneeRS}
\end{equation}

\subsection{Quantum Multiplication in $K_{q,\hbar}(\text{Hilb})$}
Okounkov and Smirnov \cite{Okounkov:2016fp} studied the operator of quantum multiplication $\textsf{M}_{\mathscr{L}}$ by line bundle $\mathscr{L}$ for an arbitrary Nakajima quiver variety $X$. This operator enters the quantum difference equation of the form
\begin{equation}
\Psi(q^{\mathscr{L}}\mathfrak{z})=\textsf{M}_{\mathscr{L}}(\mathfrak{z})\Psi(\mathfrak{z})\,,
\end{equation}
which is solved by a flat q-difference connection on functions of quantum parameter $\mathfrak{z}$ with values in $K_{\textsf{T}}$(X).
The authors study stable envelopes which can be represented as real slopes inside the second cohomology $s\in H^2(X,\mathbb{R})$. They experience a jump when $s$ crosses a rational wall in $H^2(X,\mathbb{R})$. One considers all alcoves with respect to affine hyperplane reflections for quantum Weyl group acting on $K_T(X)$. Then one can take a path from the base alcove to another one and each time we cross a wall labelled by rational slope $w_\alpha$ and define
\begin{equation}
\textsf{M}_{\mathscr{L}} = \mathcal{O}(1) \textsf{M}_{\alpha_1} \cdots \textsf{M}_{\alpha_\mathscr{L}}\,.
\end{equation}
It can be shown that the answer is path independent. In the example of Hilb$^k$ the $H^2(X,\mathbb{R})$ lattice is one-dimensional.

Operator $M_{\mathscr{L}}$ corresponds to the quantum multiplication by $\Lambda^k\mathscr{V}$ in $K_{q,\hbar}(Hilb^k)$. In order to find multiplications by $\Lambda^l\mathscr{V}\,,1\leq l<k$ one needs to make certain generalizations to \cite{Okounkov:2016fp}. 

Using results of \cite{Pushkar:2016} and \cite{Smirnov:2016},  we can formulate the following statement.
\begin{proposition}\label{eq:QuantMultConj}
The eigenvalues of quantum multiplication operators by bundles $\Lambda^l\mathscr{V}\,,1\leq l\leq k$ in quantum K-theory of $\mathcal{M}_{N,k}$ are given by symmetric polynomials $e_l(s_1,\dots s_k)$ of Bethe roots which satisfy the following Bethe equations
\begin{equation}
\prod_{l=1}^N\frac{s_a-\mathrm{a}_l}{s_a-q^{-1}\hbar^{-1}\mathrm{a}_l}\cdot\prod_{\substack{b=1 \\ b\neq a}}^k\frac{s_a-q s_b}{s_a-q^{-1}s_b}\frac{s_a-\hbar s_b}{s_a-\hbar^{-1}s_b}\frac{s_a-q^{-1}\hbar^{-1} s_b}{s_a-q\hbar s_b}=\mathfrak{z}\,,\quad a=1,\dots, k\,,
\label{eq:BetheADHM}
\end{equation}
where $\mathrm{a}_l$ are parameters from \eqref{eq:LocusN}.
\end{proposition}

The above equations can be obtained from studying saddle point behavior of the vertex function of $K_{q,\hbar}(\mathcal{M}_{N,k})$.

In particular, the eigenvalue of the multiplication by $\textsf{M}$ in \cite{Okounkov:2016fp} is given by $s_1\cdots s_k$, where the Bethe roots solve the above equations and are thus functions of all equivariant parameters and quantum parameter $\mathfrak{p}$. 

Bethe equations \eqref{eq:BetheADHM} can be easily reconstructed from a slight generalization of the ADHM quiver \figref{fig:QuiverVarieryNak} with three oriented loops (instead of usual two). It may be more convenient at this stage to make change of variables \eqref{eqt123change} and treat parameters $t_1,t_2, t_3$ democratically. 

The above Bethe equations arise as derivatives of the so-called Yang-Yang function and read
\begin{equation}
\exp \frac{\partial \mathcal{W}_{\text{ADHM}}}{\partial s_a}=1\,.
\end{equation}

\subsection{Classical Limit}
We can see that when $\mathfrak{z}=0$ equations \eqref{eq:BetheADHM} reduce to 
\begin{equation}
\prod_{\substack{b=1 \\ b\neq a}}^k(s_a-t_1 s_b)(s_a-t_2 s_b)(s_a-t_3 s_b)=0\,,\quad a=1,\dots, k\,,
\label{eq:BetheADHMp0}
\end{equation}
whose solutions provide the data for plane partitions of MacMahon modules of $\mathfrak{E}$ which we have addressed earlier. Constant $t_3$-slices of this partition reproduce symmetric functions from \eqref{eq:EUnivBundleEigenN}.

\subsection{Main Conjecture}
We have argued above that the eigenfunctions of the eRS model on discrete locus of equivariant parameters \eqref{eq:LocusN} span Hilbert space $\mathcal{H}_{nN}$. Using the results from \cite{2013arXiv1309.7094S,Koroteev:2016} we can explicitly compute the eigenvalues \eqref{eq:E1eigneeRS} of the eRS operators using the oscillator formalism. They lead us to the following statement which conjectures the relationship between the eRS spectrum and quantum K-theory $QK_{q,\hbar}(\mathcal{M}_N)$ of the moduli space of rank-$N$ sheaves on a plane.

\begin{conjecture}
Let $\mathscr{E}^{(nN)}_r,\, r=1,\dots, nN$ be the eigenvalues of eRS Hamiltonians \eqref{eq:eRSRelationsEl}. Then the eigenvalues of the operators of quantum multiplication by the $r$-th skew powers $\Lambda^r \mathscr{U}$ of the universal bundle over $\mathcal{M}_{N}$ are in one-to-one correspondence with $\mathscr{E}^{(nN)}_r$. For $r=1$ the relationship is is given by the following limit
\begin{equation}
\mathscr{E}^{(N)}_1(\mathfrak{z})=\lim_{n\to\infty}\left[\hbar^{n-1}(1-\hbar)\frac{(\mathfrak{p}\hbar;\mathfrak{p})_\infty (\mathfrak{p}q\hbar;\mathfrak{p})_\infty}{(\mathfrak{p};\mathfrak{p})_\infty (\mathfrak{p}q^{-1};\mathfrak{p})_\infty} \mathscr{E}^{(nN)}_1\right]\,,
\label{eq:eRSenergy}
\end{equation}
such that quantum parameter in $QK_{q,\hbar}(\mathcal{M}_N)$ is identified with elliptic parameter $\mathfrak{p}$ as
\begin{equation}
\mathfrak{z} = -\mathfrak{p}\sqrt{q\hbar}\,.
\label{eq:matchingzp}
\end{equation}
\end{conjecture}

In other words, there is a stable $n\to\infty$ limit of eRS energies which returns eigenvalues of the operator of quantum multiplication in  $\mathcal{M}_{N}$. The function with ratios of infinite q-Pochammer symbols in the above formula corresponds to a $U(1)$ character on the moduli space of torsion-free rank-$nN$ sheaves.

We hope that this conjecture will be proven in the near future along the lines of \cite{Okounkov:2004a} by Okounkov and Pandharipande. 
In was proven in \cite{2013arXiv1309.7094S} that
\begin{equation}
\frac{\theta_1(\hbar x|\mathfrak{p})}{\theta_1(x|\mathfrak{p})} = F(\mathfrak{p})\sum_{n\in\mathbb{Z}}\frac{x^n}{1-\hbar\mathfrak{p}^n}\,,
\end{equation}
where $F(\mathfrak{p})$ is a known function. This observation can be used to compare term-by-term expansions of the eRS eigenproblem with integrals of the moduli spaces of genus zero quasimpas to $\mathcal{M}_N$.

\subsection{Example}
The first nontrivial example is $\mathcal{M}_{2,1}=\text{Hilb}^2$. As we already discussed in the example after Proposition \ref{Prop:Eigenvalues}, there are two states with $k=2$ corresponding to $\tiny\yng(1,1)$ and $\tiny\yng(2)$ tableaux. The eigenvalue of quantum multiplication by bundle $\Lambda^2\mathscr{U}$ in $QK_{q,\hbar}(\mathcal{M}_1)$ is as follows
\begin{equation}
\mathscr{E}_2=1-(1-q)(1-\hbar)s_1 s_2\,,
\end{equation}
where Bethe roots satisfy the following equations
\begin{equation}
\frac{\left(s_1-1\right)}{\left(s_1-q^{-1}\hbar^{-1}\right)}\cdot\frac{\left(s _1-q s _2\right)\left(s _1-\hbar s _2\right)\left(s _1-q^{-1}\hbar^{-1}s_2\right)}{\left(s _1-\hbar^{-1} s_2\right) \left(1- q\hbar^{-1} s_1\right)\left(s _1-q\hbar s
   _2\right)}=\mathfrak{p}
\label{eq:Betherk2}
\end{equation}
and the second equation with $s_2$ and $s_1$ interchanged. We can solve the above equations perturbatively in $\mathfrak{p}$ and extract the value of $\mathscr{E}_2$ as a series. It was verified in \cite{Koroteev:2018a} that this expression is in agreement with the eRS energy according to 
\eqref{eq:eRSenergy}.

Let us now find the eigenvalues of operator $\textsf{M}$ on Hilb$^2(\mathbb{C}^2)$. In this case it reads (see Sec 8.3.7 of \cite{Okounkov:2016fp}) $\textsf{M}=\mathscr{O}(1) B_0 B_{1/2}$ and its first eigenvalue has the following expansion in quantum parameter
\begin{align}
&s_1=t_1-\frac{t_1 \left(t_1+1\right) \left(t_2-1\right) \left(t_1 t_2-1\right)
   }{t_1-t_2}\mathfrak{z}
   +t_1 \Big(-\frac{\left(t_1 t_2-1\right){}^3}{t_1-t_2}+t_1 t_2 \left(3 t_1 t_2-1\right)\cr
   &-\frac{t_2 \left(\left(t_2^4-2
   t_2^2+2\right) t_1^4-2 t_2 t_1^3-\left(2 t_2^4-4 t_2^2+2\right) t_1^2+2 t_2^3 t_1-2
   t_2^2+1\right)}{\left(t_1-t_2\right){}^3}\Big) \mathfrak{z}^2+\dots\,,
\end{align}
while the second eigenvalue $s_2$ is given by replacing $t_1\to t_2$. Both eigenvalues coincide individually with product $s_1s_2$, where $s_1$ and $s_2$ satisfy Bethe equations \eqref{eq:Betherk2} contingent to $t_1=q,\,t_2=\hbar$ and \eqref{eq:matchingzp}.

\subsection{DT/PT/GW Correspondence}
We would like to formulate the following conjecture for Hilbert schemes:
\begin{conjecture}
K-theoretic DT invariants on $\mathbb{P}^1\times\mathbb{C}^2$ are in one-to-one correspondence with quantum K-theory of Hilbert scheme of points
\begin{equation}
KDT(\mathbb{P}^1\times\mathbb{C}^2) \longleftrightarrow QK_{\textrm{T}}(Hilb)\,.
\end{equation}
\end{conjecture}

Similarly, for $N>1$ there must be a generalization of \cite{Diaconescu:2008} to K-theory relating higher rank Donaldson-Thomas invariants with ADHM moduli sheaves. More work needs to be done to properly define and prove this correspondence.

\appendix
\section{tRS Difference equation for $T^\ast \mathbb{P}^1$}\label{Sec:tRSP1}

Let us illustrate \eqref{eq:tRSclass} for $X_2=T^\ast \mathbb{P}^1$. 
The vertex function is given by
\begin{equation}
V=\frac{e^{\frac{\log\zeta_2\cdot\log a_1\cdots a_n}{\log q}}}{2\pi i}\int\limits_C \frac{ds}{s}\,e^{\frac{\log\zeta_1/\zeta_2\cdot\log s}{\log q}} \frac{\varphi\left(\hbar\frac{s}{a_1}\right)}{\varphi\left(\frac{s}{a_1}\right)}\frac{\varphi\left(\hbar\frac{s}{a_2}\right)}{\varphi\left(\frac{s}{a_2}\right)}\,.
\end{equation}
For $X_2$ there are two tRS operators
\begin{align}
T_1(\boldsymbol{\zeta}) &= \frac{\hbar\zeta_1-\zeta_2}{\zeta_1-\zeta_2} p_1 + \frac{\hbar\zeta_2-\zeta_1}{\zeta_2-\zeta_1} p_2\,,\cr
T_2(\boldsymbol{\zeta}) & = p_1 p_2\,.
\end{align}
By acting with these operators on the above function we get
\begin{align}
T_1(\boldsymbol{\zeta}) \mathrm{V} &= \mathrm{V}^{\left(T_1(s)\right)}\,,\cr
T_2(\boldsymbol{\zeta}) \mathrm{V} & = a_1 a_2 \mathrm{V}\,,
\end{align}
where the first tRS class is given by\footnote{We abusing the notation by denoting by $T_r$ both the tRS class and the operator. Hopefully this will not confuse the reader.}
\begin{equation}
T_1(s) = \frac{\hbar\zeta_1-\zeta_2}{\zeta_1-\zeta_2} s + \frac{\hbar\zeta_2-\zeta_1}{\zeta_2-\zeta_1} \frac{a_1 a_2}{s}\,,
\end{equation}
and is a linear combination of the tautological bundle  $\mathscr{V}$ and  $\Lambda^2\mathscr{W}\otimes\mathscr{V}^\ast$ over $X_2$ with coefficients dependent on quantum parameter $z=\frac{\zeta_1}{\zeta_2}$ and the equivariant parameters.

In order to prove that $\mathrm{V}^{\left(T_1(s)\right)}=(a_1+a_2)\mathrm{V}^{(1)}$ we shall use the following integral (see also \cite{Nedelin:2017aa}, Appendix D)
\begin{equation}
I_2=\int\limits_C \frac{ds}{s}\left[\frac{1}{s}\left(1-\frac{s}{a_1}\right)\left(1-\frac{s}{a_2}\right)\right]\,e^{\frac{\log z \cdot\log s}{\log q}} \frac{\varphi\left(\hbar\frac{s}{a_1}\right)}{\varphi\left(\frac{s}{a_1}\right)}\frac{\varphi\left(\hbar\frac{s}{a_2}\right)}{\varphi\left(\frac{s}{a_2}\right)}\,,
\end{equation}
where contour $C$ is chosen in such a way that shift $s\to q^{-1}s$ does not pick up any poles. This can be straightforwardly generalized to the $n$-particle tRS model. Thus integral $I_2$ in the shifted variable is equal to itself which leads us to 
\begin{equation}
0=\int\limits_C \frac{ds}{s}\frac{1}{z \hbar}\left[(1-\hbar z)s+(\hbar- z)\frac{a_1 a_2}{s}-(a_1+a_2)(1-z)\right]\,e^{\frac{\log z \cdot\log s}{\log q}} \frac{\varphi\left(\hbar\frac{s}{a_1}\right)}{\varphi\left(\frac{s}{a_1}\right)}\frac{\varphi\left(\hbar\frac{s}{a_2}\right)}{\varphi\left(\frac{s}{a_2}\right)}\,,
\end{equation}
from where the statement follows since the expression in the square brackets must vanish.

\bibliography{cpn1}

\end{document}